\documentclass[11pt]{article}
\usepackage[utf8]{inputenc}
\usepackage{xcolor}
\usepackage[english]{babel}
\usepackage{amsmath}
\usepackage{amsthm}
\usepackage{amsfonts}
\usepackage{amsthm}
\usepackage{esint}
\usepackage{amssymb} 
\usepackage{comment}

\usepackage{a4wide} 		
\usepackage{graphicx}
\usepackage{lmodern}
\usepackage[T1]{fontenc}
\usepackage[affil-it]{authblk}
\usepackage{epstopdf}
\newtheorem{definition}{Definition}[section]
\usepackage[linkcolor=red,colorlinks=true]{hyperref}
\let\originaleqref=\eqref
\numberwithin{equation}{section}
\renewcommand{\eqref}{equation~\originaleqref}

\newtheorem{theorem}{THEOREM}[section]
\newtheorem{lemma}[theorem]{Lemma}

\title{Approximation of rigid obstacle by highly viscous fluid}
\author{Sadokat Malikova}
\affil{University of Warsaw,\\Institute of Applied Mathematics and Mechanics,\\  ul. Banacha 2, Warsaw, Poland\\ \smallskip sadokat.malikova@mimuw.edu.pl}
\date{}

\begin{document}

\maketitle
\begin{abstract} {In this paper, we study the problem concerning the approximation of a rigid obstacle for flows governed by the stationary Navier-Stokes equations in the two-dimensional case. The idea is to consider a highly viscous fluid in the place of the obstacle.
Formally, as the fluid viscosity goes to infinity inside the region occupied by the obstacle, we obtain the original problem in the limit.

The main goal is to establish a better regularity of approximate solutions. In particular, the pointwise estimate for the gradient of the velocity is proved.
 We give numerical evidence that the penalized solution can reasonably approximate the problem, even for relatively small values of the penalty parameter. 

  }
\end{abstract}

\section{Introduction}

We are interested in approximation of rigid obstacle by highly viscous fluid. We prove that the weak solution to the approximate problem has some higher regularity. Moreover, we claim that the approximate problem tends to be a rigid obstacle problem in the high viscosity limit. Finally, we illustrate our claims using numerical simulations, showing that our approach is applicable in numerics. This approach should in particular enable us to approximate rough obstacles. We show numerical results for such case, leaving its theoretical analysis for future work. 
  
   Let $\Omega \subset \mathbb{R}^2 $ be an open bounded domain of class $C^2$. Suppose that $\Omega \setminus S $ filled with a homogeneous viscous incompressible fluid with an obstacle inside (Fig.\ref{fig:Flow1}). Denote by $ S \subset\Omega$ the domain occupied by the rigid obstacle, $f$ is a given vector function. To simplify calculations we assume the density $\rho=1$. 

Our approach to the problem is based on the penalization method. We represent rigid obstacle as a highly viscous fluid. Consider an approximate stationary Navier-Stokes system 
\begin{equation}\label{NS}
 (u\cdot\nabla)u -{\rm div\,}\left[\nu(x)\mathbf{D} u\right] +\nabla{p} =f
\end{equation}
\begin{equation}\label{NSdiv} 
{\rm div\,}u=0
\end{equation}
that satisfies the boundary condition
\begin{equation}\label{eq:bc}
u=0 \,\,\textrm{at}\,\, \partial\Omega ,
\end{equation}
where the kinematic viscosity $\nu(x)$ is a discontinuous function that has the following structure
\begin{equation}\label{visc}
\nu(x)= \begin{cases}
1, & x\in{\Omega\setminus{S}},\\
m, & x\in{S},
\end{cases}
\end{equation} 
and $\mathbf{D}u$ is the deformation rate tensor with the components:  
 \[\mathbf{D_{ij} }u=\frac{1}{2}\left( \frac{\partial u_i}{\partial x_j}+\frac{\partial u_j}{\partial x_i}\right). \]

 Since we interested in weak solutions, we introduce integral formulation of the problem:
  we say that a vector function $u\in V(\Omega)$ is a solution to the problem if the following integral identity 

\begin{equation}\label{weakNS}
\int_{\Omega}((u\cdot\nabla)u)w\,dx+\\
\int_{\Omega}\nu{\mathbf{D}(u)}:\mathbf{D}(w)\,dx=\int_{\Omega} f\cdot w\,dx
\end{equation}
holds true for arbitrary $w  \in V(\Omega)$, where $V$ is defined at the beginning of Section 2.

In this paper, we want to prove that the gradient of the velocity field of the approximate problem (\ref{NS})-(\ref{NSdiv}) has a pointwise estimate in $L^{\infty}$ norm.


The first method of our study is the penalty method that we gained from V. Starovoitov's work [\ref{staravoitov}]. The method considers the rigid obstacles as fluids whose viscosity tends to infinity. The author presented the application of the method to the classical problem of rigid obstacles in a viscous incompressible fluid. In this direction, we mention here the papers [\ref{hoffman}] by K.H. Hoffmann and V. Starovoitov, [\ref{sanmartin}] by A. San Martin, V. Starovoitov, and Tucsnak, [\ref{Aneta}] by A. Wr\'oblewska-Kami\'nska, who used the penalty method to prove the global existence of weak solutions. 
 
  Also, we use tangential regularity techniques that are borrowed from J.-Y. Chemin [\ref{chemin}], R. Danchin, F. Fanneli and M. Paicu [\ref{danchin}]. They proved the existence and of weak solutions for the problem with the jump density for the barotropic compressible Navier-Stokes equations. In the case where the density has tangential regularity with respect to some non-degenerate family of vector fields, they also have uniqueness. Furthermore, this method helps us to approximate non-smooth domains with information about the geometry of the obstacle.

  Moreover, we aimed to show that (\ref{NS})-(\ref{visc}) may have some practical applications in numerics. We mention here M.Dryja's work [\ref{dryja}], where he analyzed the elliptic problem with highly discontinuous coefficients. The author shows that the convergence of the method (DG) presented in the work is almost optimal and only weakly depends on the jumps of coefficients.

\begin{figure}
\centering
  \includegraphics[width=80mm]{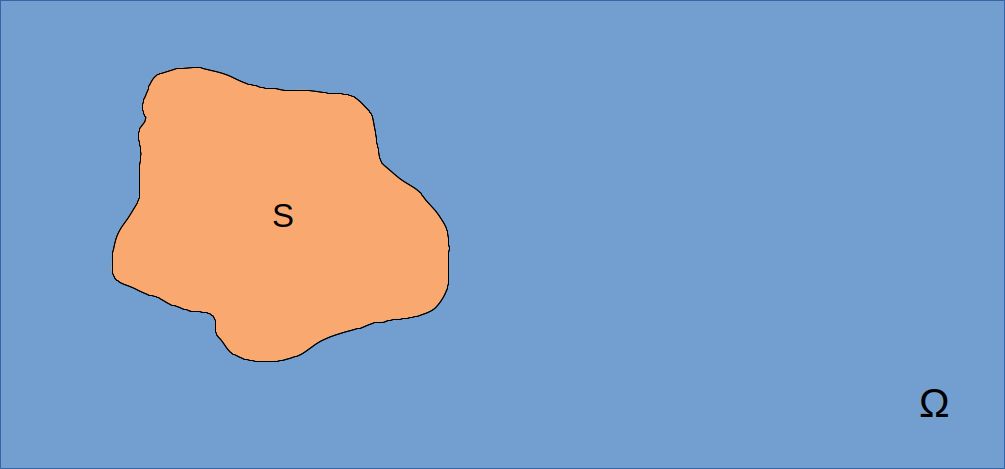}
  \caption{Fluid and rigid obstacle domains.}
  \label{fig:Flow1}
\end{figure}

 The paper is organized as follows. In section 2 we state our main results, Theorem \ref{thm_exis} (existence of weak solutions), Theorem \ref{limit} (the limit case), Theorem \ref{thm_main} (higher regularity of approximate solution) and introduce some general notations. Section 3 is devoted to the analysis of the Stokes system. In subsection 3.1  we introduce a regularity result for the model case, and the higher regularity of the Stokes system is proved in subsection 3.2.
   Section 4 is the core of this paper. There we introduce the tangential regularity results that are needed in the proof of Theorem 2.3. In section 5 we present numerical evidence that the approximate problem may have some practical applications.  
   
   The proof of existence of weak solutions and their high viscosity limit (Theorems 2.1  and 2.2) are postponed in the Appendix. These results are based on the Bogovskii type approach and general theory.

\section{Notations and Main results }

In this section, we introduce some more general notation. In order to define spaces of divergence-free vector functions we introduce
$$
\mathcal{V} =\{ v\in {{C}_{0}^{\infty}(\Omega,\mathbb{R}^{2})}\,\vert\,{{\rm div\,}{v}=0}\}, \quad
V(\Omega)\; \textrm{is the closure of} \; \mathcal{V}(\Omega)\; {\rm in}\; H^{1}_0(\Omega).
$$

According to classical result [\ref{temam}] for $\Omega$ an open Lipschitz set we have
$$V(\Omega)=\lbrace v\in{H^{1}_{0}}\,\vert\, {\rm div\,} v =0 )\rbrace .$$
Next we define
$$ 
H(\Omega)=\lbrace v\in L^{2}\,\vert\, {\rm div\,} v =0\;\textrm{in}\; \mathcal{D}^{'}(\Omega),\; v\cdot n =0\; \textrm{in}\; H^{-\frac{1}{2}}(\partial\Omega)\rbrace
.$$
The space $H$ is equipped by scalar product $(\cdot,\cdot)$, and the space $V$ is the Hilbert space with the scalar product 
 \[ ((u,v))=\sum_{i=1}^{n}(\nabla_{i}u,\nabla_{i}v).\] \\
By $"\cdot"$ we denote the scalar product of two vectors,
\[\xi\cdot \eta =\sum_{i=1}^{n}\xi_{i}\eta_{i},\]
for ${\xi}=\{\xi_{1},\xi_{2},...,\xi_{n}\}\in{\mathbb{R}^n}$ and ${\eta}=\{\eta_{1},\eta_{2},...,\eta_{n}\}\in{\mathbb{R}^n}$  and ":" stands for scalar product for two tensors,
\[ \xi : \eta =\sum_{i,j=1}^{n}\xi_{i,j}\eta_{i,j},\]\\
for $\xi =\{\xi_{i,j}\}_{i=1,...,n,j=1,...,n}\in{\mathbb{R}^{n\times{n}}}$ and $\eta =\{\eta_{i,j}\}_{i=1,...,n,j=1,...,n}\in{\mathbb{R}^{n\times{n}}}$.\\
Moreover, $f\in L^2(\mathbb{R}_{x_2};L^2(\mathbb{R}_{x_1}))$ stands notation of 1-dimensional space in $x_1$ and $x_2$ directions respectively, and  $f\in H^1 (\mathbb{R}_{x_2};H^1 (\mathbb{R}_{x_1}))$ means that $f, f_{x_2}\in L^2(\mathbb{R}_{x_2};H^1 (\mathbb{R}_{x_1}))$ i.e.
\begin{equation}\label{1-dimsp}
 \Vert f \Vert_{L^2(\mathbb{R}_{x_2};L^2 (\mathbb{R}_{x_1}))}:=\left[ \int_{\mathbb{R}_{x_2}}\left( \int_{\mathbb{R}_{x_1}}\vert f(x_{1},x_{2})\vert^2\, dx_{1}\right) dx_{2}\right]^{\frac{1}{2}} 
\end{equation}

\begin{equation}\label{1-Dsp}
\Vert f \Vert_{L^2(\mathbb{R}_{x_2};H^1 (\mathbb{R}_{x_1}))}:=\left[ \int_{\mathbb{R}_{x_2}}\left( \int_{\mathbb{R}_{x_1}}\vert f(x_{1},x_{2})\vert^2+\vert f_{x_1}(x_{1},x_{2})\vert^2\, dx_{1}\right) dx_{2}\right]^{\frac{1}{2}}  .
\end{equation}


Due to assumed regularity of $\Omega$ and $S$, there exists a vector field $X\in C^2$, such that  
\begin{equation} \label{X}
X\cdot\tau = 1, \; X\cdot n = 0 \quad {\rm on}\; \partial S  \cup \partial \Omega. 
\end{equation}

 The first result concerns existence of weak solutions [\ref{temam}] reads 
\begin{theorem}[Existence]\label{thm_exis}
Let $\Omega \subset \mathbb{R}^2 $ be an open bounded domain of class $C^2$, let f be given in $L^2(\Omega)$. Then Problem (\ref{weakNS}) has at least one solution $u\in V(\Omega)$ and there exist a function $p\in L^2(\Omega)$ such that  (\ref{NS}) are satisfied. 
\end{theorem}

 The next result is about the limit $m\longrightarrow \infty$.
\begin{theorem}[The limit]\label{limit}
Let assumptions of Theorem \ref{thm_exis} hold and let $m\longrightarrow \infty$, then $\mathbf{D}u=0$  in the domain $S$. 
\end{theorem}

When $\mathbf{D}u=0$ in $S$ we obtain rigid motion. The regularity of $u$ implies that the trace of $u\vert_{\partial S}= \phi $ is $\phi \in H^{\frac{1}{2}}(\partial S)$.  In situation when the obstacle touches the boundary, (\ref{eq:bc}) together with Theorem \ref{limit} implies $\Vert \phi \Vert_{H^{\frac{1}{2}}(\partial S)}\longrightarrow 0$.  

The higher regularity of the approximate problem plays a crucial role in the proof of uniqueness.  Generally, if the solution has tangential regularity with respect to some vector field, we get uniqueness. The main result of this paper is stated as follows

\begin{theorem}\label{thm_main} Let $\Omega \subset \mathbb{R}^2 $ be an open bounded domain of class $C^2$, $S\subset \Omega$ with the boundary $\partial{S}\in C^2 $, and let  ${f}\in H^{1} (\Omega)$, then for every solution $u$ of Navier-Stokes equations (\ref{NS})-(\ref{NSdiv}) we have
 $\nabla u\in L^{\infty}(\Omega)$.
\end{theorem}
Due to technical difficulties, we split our analysis into two parts. We consider two approximate problems: Stokes and Navier-Stokes.
The method of our proof relies on tangential regularity results for the approximate problem. The proof of tangential regularity differs from Danchin [\ref{danchin}] and Chemin' [\ref{chemin}] works. The difficulty is that we propagate the whole approximate Navier-Stokes equations along the given vector field.

  We prove tangential regularity for the Stokes system, that is stated in Lemma \ref{lem_trSt}, where we derive energy estimate for $\partial_{X} u$.
In Lemma \ref{l_hihgerRegSt} we state the higher tangential regularity of Stokes system. In other words, we differentiate twice the whole system of equations along the given vector field $X$ and prove that $\nabla \partial_{X}^2 u\in L^2(\Omega)$. In the proof of tangential regularity results, we apply the Bogovskii type approach [\ref{galdi_book}] to the 'pressure terms' $\partial_{X}p$, $\partial^2_{X} p$.
   
   In subsection 3.1, we consider a model case where the problem reduces to $1$-dimensional functional spaces, which allows us to prove the higher regularity of a solution for the approximate Stokes system. The proof requires tangential regularity results and elementary tools like H\"older, Poincar\'e inequalities and embedding. 
    We apply the model case idea in the proof of Theorem \ref{thm_main}.

    We establish tangential and higher tangential regularity results for the approximate Navier-Stokes system (\ref{NS})-(\ref{NSdiv}) in Lemma \ref{l_PropNS} and Lemma \ref{higtan}. Proofs of these lemmas require Lemma \ref{lem_trSt} and Lemma \ref{l_hihgerRegSt}.


The interesting part of our approach is the viscosity jump area. Thus, in the proof of Theorem \ref{thm_main}, we concentrate our analysis on $\Sigma$ domain which is the neighbourhood of the approximate obstacle boundary. We straighten out the boundary using a transition to the curvilinear coordinate system ([\ref{takahashi}]). We follow the model case idea. Tangential regularity results are the main steps in the proof of pointwise estimate for the gradient of the velocity field in $L^{\infty}$ norm.
We apply Lemma \ref{l_PropNS}, Lemma \ref{higtan} and use tools like H\"older, Poincar\'e inequalities and embedding.   

The proofs of Theorem \ref{thm_exis} and Theorem \ref{limit} are the most standard and rely on known results(see Appendix).

\section{Regularity for Stokes equations}
 In this section, we consider the approximate Stokes system of equations in a two dimensional case
\begin{equation}\label{stokes_eq}
-{\rm div\,}\left[\nu(x)\mathbf{D}u\right] +\nabla{p}=f
\end{equation}
\begin{equation}\label{div_eq}
{\rm div\,}u=0. 
\end{equation}

Recall, that viscosity is a jump function (\ref{visc}).
The system of equations (\ref{stokes_eq})-(\ref{div_eq}) appends the condition at the boundary, that is,
\begin{equation}\label{eq:bcSt}
u=0\, \, \textrm{at}\,\, \partial\Omega .
\end{equation}

\begin{definition} \label{def_Stokes}
A vector field $u:\Omega\longrightarrow \mathbb{R}^2 $ is called a weak (or generalized) solution to the Stokes problem (\ref{stokes_eq})- (\ref{div_eq}) if and only if $u\in V(\Omega)$ and it verifies the identity
\begin{equation}\label{eq:1.6_eq}
(\nu\mathbf{D} u,\mathbf{D}\phi)=-(f,\phi),\;\;\,\forall\,\phi \in V(\Omega) 
.\end{equation}
\end{definition}

The following lemma gives tangential regularity of the solution to Stokes problem, which will be useful in the proof of Lemma \ref{l_PropNS}. 
\begin{lemma}\label{lem_trSt} Let $\Omega \subset \mathbb{R}^2 $ be an open bounded domain of class $C^2$, $S\subset \Omega$ with the boundary $\partial{S}\in C^2 $. 
Assume $X$ satisfies (\ref{X}) and 
$f,\partial_X f \in L^2 (\Omega)$. 
Then for every weak solution $u$ of (\ref{stokes_eq})-(\ref{div_eq}) we have
\begin{equation}\label{eq:tan_regSt}
\Vert\nabla \partial_X u\Vert_{L^2}\leqslant C_1 \left( \Vert f\Vert_{L^2}+\Vert \partial_X f\Vert_{L^2} \right) 
\end{equation}
where  $C_1 := C_1 (\Omega ,X)$. 
\end{lemma}
\begin{proof}

  Assume that a given vector-field $X=(X^1,X^2)$ is sufficiently smooth, and we are interested in the regularity of function $u$ along $X$, i.e. in the quantity
  \[\partial_X u:=X^1 \partial_{x_1} u + X^2 \partial_{x_2} u .\]
  We take the derivative along tangential vector-field from Stokes equation, that is
  \begin{equation}\label{eq:st_eq}
  -\partial_X {\rm div\,}\left[\nu(x)\mathbf{D}u\right] +\partial_X \nabla{p}=   \partial_X f .
  \end{equation}
  
  In order to get an equation for $\partial_{X} u$, we are going to rewrite (\ref{eq:st_eq}). After some  calculations we get  

\begin{equation}\label{eq:1term_eq}
\partial_X {\rm div\,}\left[\nu(x)\mathbf{D}u\right]={\rm div\,}\left[\nu(x)\partial_X \mathbf{D}u\right]-\sum_{i}X^i_{,k}\partial_{x_i} \left(  \nu(x)\mathbf{D}u\right) 
.\end{equation}

 We have
\[\nabla({\partial_X u})=\nabla{u}\nabla{X}+\partial_X \nabla u .\]
Using the above formula we rewrite the directional derivative of symmetric tensor  as follows
\begin{equation}\label{3.6_eq}
 \partial_X (\mathbf{D}u)=\frac{1}{2}\left[\nabla({\partial_X u})+ \nabla^T ({\partial_X u})-\nabla{u}\nabla{X}-\nabla{X}^T \nabla{u}^T\right] 
 .\end{equation}
	Finally, using (\ref{3.6_eq}) and (\ref{eq:1term_eq}) we rewrite \eqref{eq:st_eq} in the following  way  
\begin{equation}
\begin{split}\label{eq:3.7_eq}
&-{\rm div\,}\left[\nu (x) \mathbf{D}(\partial_X u) \right] +\frac{1}{2}{\rm div\,}(\nu(x)\nabla^T X \nabla^T u+\nu(x)\nabla u \nabla X)\\
&+\sum_{i}X^i_{,k}\partial_{x_i}\left( \nu(x)\mathbf{D}u\right) +\nabla (\partial_X p)-\sum_{i}X^i_{,k}\partial_{x_i} p=\partial_X f
.\end{split}
\end{equation} 
Multiplying the \eqref{eq:3.7_eq} by $\psi\in C^{\infty}_0(\Omega) $ and integrating by parts we obtain 
 
\begin{equation}\label{eq:3.8}
\begin{split}
&\int_{\Omega} \nu (x)\mathbf{D}(\partial_X u):\mathbf{D}(\psi) dx-\frac{1}{2}\int _{\Omega} \nu(x)\left(  \nabla^T X \nabla^T u +\nabla u \nabla X\right) : \nabla(\psi) dx\\
& -\int_{\Omega}\left( \nu(x)\mathbf{D}u\right)\partial_{x_i}(X^i_{,k}\,\psi) \,dx-\int_{\Omega} \partial_X\left(  p\right) \, {\rm div\,}\psi\,dx+\int_{\Omega} p\,\partial_{x_i}\left( X^i_{,k} \psi\right) \,dx=\int_{\Omega} \partial_X f\,\psi \,dx
.\end{split}
\end{equation}
We test \eqref{eq:3.7_eq} by function $\partial _X v \in H^{1}_0 (\Omega)$ and get  
\begin{equation}\label{eq:weak_eq}
\begin{split}
\int_{\Omega}&\nu(x)\mathbf{D}(\partial_X u):\mathbf{D}(\partial_X v) dx-\frac{1}{2}\int _{\Omega} \nu(x)\left(  \nabla^T X \nabla^T u +\nabla u \nabla X\right) : \nabla(\partial_X v) dx\\
 &-\int_{\Omega}\left( \nu(x)\mathbf{D}u\right)\partial_{x_i}(X^i_{,k}\,\partial_X v) \,dx-\int_{\Omega} \partial_X\left(  p\right) \, {\rm div\,}(\partial_X v)dx\\
 &-\int_{\omega}X^i_{,k}\partial_{x_i} p\,\partial_X v \,dx=\int_{\Omega} \partial_X f\,\partial_X v \,dx
.\end{split}
\end{equation}

Note that, ${\rm div\,}u=0$ , but as we are differentiating the system of Stokes equations along the given vector field $X$ we get 
\begin{equation}\label{div_X}
0= \partial_X{\rm div\,}u={\rm div\,}(\partial_X u)-\sum_{i,j=1}^{2}\partial_{x_i} X^j \partial_{x_j} u^i .
\end{equation}
 From the above equality we deduce that  ${\rm div\,}(\partial_X u)=\sum_{k,j=1}^{2}\partial_{x_k} X^j \partial_{x_j} u^k$.\\

\subparagraph{The pressure term.} 
We are going to show that $\partial_X p\in L^2 (\Omega) $. In order to estimate $\partial_X p $ we will use the Bogovskii type approach (see Lemma \ref{lem_bogovskii}). 
Our aim is to show 
\begin{equation}\label{eq:bogov2_eq}
\begin{split}
\Vert \partial_{X} p\Vert_{L^2} &\leqslant C_{b_1}\Big(\Vert\nu (x)\mathbf{D}\partial_X u\Vert_{L^2}  +\frac{1}{2}\Vert\nu(x)( \nabla^T X \nabla^T u +\nabla u \nabla X)\Vert_{L^2} \\
&+\left((\Vert \nu(x)\mathbf{D}u\Vert_{L^2} + \Vert p\Vert_{L^2}\right)(c_p c_{X"}+c_{X'})+c_p \Vert\partial_X f\Vert_{L^2} \Big)
,\end{split}
\end{equation}
where $c_p $ is the constant from Poincar\'e inequality and $c_{X"}, \,c_{X'}$ are constants depending only on $X$.
In general,   $\partial_X p $ satisfies a following inequality
\[ \left| \int_{\Omega} \partial_{X}p\, dx\right| =\left|  -\int_{\Omega} p\,{\rm div\,}X  \,dx +\int_{\partial\Omega} X^i n^i p\, ds\right| \leqslant \Vert{\rm div\,}X\Vert_{L^2} \Vert p\Vert_{L^2} .\;\;\;\, (*)\]
The above inequality we get via integrating by parts and using the condition on the boundary that is
\[X^i n^i =0 \,\, on \,\, \partial\Omega .\] 

Recall that (\ref{eq:st_eq}) is equivalent to (\ref{eq:3.7_eq}).
Let us consider the functional

\begin{equation}\label{fnctl}
\begin{split}
\mathcal{F}(\psi) &=(\nu (x)\mathbf{D}(\partial_X u),\mathbf{D}\psi) -\frac{1}{2}\left(\nu(x) ( \nabla^T X \nabla^T u +\nabla u \nabla X), \nabla\psi\right)\\
& -\left( \nu(x)\mathbf{D}u,\partial_{x_i}(X^i_{,k}\,\psi)\right)+(p,\partial_{x_i}( X^i_{,k} \psi))-(\partial_X f,\,\psi)
 \end{split}
\end{equation}

for all $\psi\in H^{1}_{0} (\Omega)$.
Recall (\ref{eq:2.13_eq}), if take a test function  $ \partial_{X}^{*}\psi \in C^{\infty}_0(\Omega)$, we get
 \[(p,{\rm div\,}\partial_{X}^{*}\psi)=-(\partial_X  p \, ,{\rm div\,}\psi)+\Big( p\, ,\partial_{x_i}\left( X^i_{,k} \psi\right)\Big) \]
 where $ \partial_{X}^{*}\psi =-\partial_{x_i}(X^i \, \psi ) $.
 Therefore by Lemma \ref{lem_bogovskii} there exists a uniquely determined  $\partial_{X} p\in L^2 (\Omega)$ that  $ \frac{1}{\vert \Omega \vert}\int_{\Omega} \partial_{X}p$ is bounded, and such that 
\begin{equation}\label{3.10}
\mathcal{F}(\psi)=(\partial_{X}p,{\rm div\,}\psi)
.\end{equation}
Consider the problem
\[{\rm div\,}\psi= \partial_{X} p- \frac{1}{\vert \Omega \vert}\int_{\Omega} \partial_{X}p =g \]
\begin{equation}\label{eq:sys2_eq}
\psi \in H^{1}_0 (\Omega)
\end{equation}
\[\Vert\psi\Vert_{H^{1}}\leqslant C_b \Vert \partial_{X} p\Vert_{L^2},\]

with $\Omega$ bounded and satisfying the cone condition. Since
$$\int_{\Omega}g=0 ,\, g\in L^2(\Omega),\;\,\,\Vert\psi\Vert_{H^{1}}\leqslant C_{b1} \Vert \partial_{X} p\Vert_{L^2}$$   
from Theorem III.3.1 ([\ref{galdi_book}]) we deduce the existence of $\psi$ solving the  \eqref{eq:sys2_eq}. We use such a $\psi$ as a test function in (\ref{3.10}), we obtain

\begin{equation}
\begin{split}
\Vert \partial_{X} p\Vert_{L^2}^2 &=(\nu (x)\mathbf{D}(\partial_X u),\mathbf{D}\psi) -\frac{1}{2}\left(\nu(x) ( \nabla^T X \nabla^T u +\nabla u \nabla X), \nabla\psi\right)\\
& -\left( \nu(x)\mathbf{D}u,\partial_{x_i}(X^i_{,k}\,\psi)\right)+(p,\partial_{x_i}( X^i_{,k} \psi))-(\partial_X f,\,\psi)
.\end{split}
\end{equation}
Applying the H\"older and Poincar\'e inequalities to the above equation, we get
\begin{equation}
\begin{split}
\Vert \partial_{X} p\Vert_{L^2}^2 &\leqslant \Vert\nu (x)\mathbf{D}\partial_X u\Vert_{L^2} \Vert\mathbf{D}\psi\Vert_{L^2} +\frac{1}{2}\Vert( \nabla^T X \nabla^T u +\nabla u \nabla X\Vert_{L^2} \Vert \nabla\psi \Vert_{L^2}\\
& +\Vert \nu(x)\mathbf{D}u\Vert_{L^2} \Vert X^i_{,ki}\,\psi +X^i_{,k}\,\nabla\psi\Vert_{L^2} +\Vert p\Vert_{L^2} \Vert X^i_{,ki}\,\psi +X^i_{,k}\,\nabla\psi\Vert_{L^2} +\Vert\partial_X f\Vert_{L^2} \Vert\psi\Vert_{L^2}\\
&\leqslant\Vert\nu (x)\mathbf{D}\partial_X u\Vert_{L^2} \Vert\nabla\psi\Vert_{L^2} +\frac{1}{2}\Vert\nu(x)( \nabla^T X \nabla^T u +\nabla u \nabla X)\Vert_{L^2} \Vert \nabla\psi \Vert_{L^2}\\
&+ (\Vert \nu(x)\mathbf{D}u\Vert_{L^2} +\Vert p\Vert_{L^2} )(c_p c_{X"}+c_{X'})\Vert \nabla \psi \Vert_{L^2} +c_p \Vert\partial_X f\Vert_{L^2} \Vert\nabla\psi\Vert_{L^2}
.\end{split}
\end{equation}
Using inequality (\ref{eq:sys2_eq}) we reduce both sides of the above expression by the term $\Vert \partial_{X} p\Vert_{L^2} $, and get the required estimate (\ref{eq:bogov2_eq}).

Now,  we will examine the remaining terms of the equation (\ref{eq:weak_eq}), in order to estimate them. Recall that $X\cdot n =0$, $X\cdot \tau =1 $. We rewrite the directional derivative as
\[ \partial_{X} u= X\cdot \tau\, \partial_{\tau} u +X\cdot n\, \partial_{n} u .\]
By assumptions and boundary condition (\ref{eq:bcSt}), we get
\begin{equation}\label{bc}
 \partial_{X} u \vert_{\partial \Omega}=  \partial_{\tau} u \vert_{\partial{\Omega}} =0
.\end{equation}
Then Korn inequality holds 
\begin{equation}\label{eq:korn2_eq}
\int_{\Omega} \nu (x)\vert\mathbf{D}(\partial_X u)\vert^2 dx \geqslant C\int_{\Omega} \vert\nabla\partial_X u\vert^2 dx
.\end{equation}
Using H\"older and Young's inequalities to the 2nd term of \eqref{eq:weak_eq}, we get
\begin{equation}
\begin{split}
\frac{1}{2}\int _{\Omega} & \nu(x)\left(  \nabla^T X \nabla^T u +\nabla u \nabla X \right): \nabla(\partial_X v) dx\\
&\leqslant \int _{\Omega} \nu(x)\left( \vert\nabla u\vert\vert \nabla X \vert\right) \vert\nabla(\partial_X v)\vert dx\\
&\leqslant C \Vert \nu(x)\nabla u\Vert_{L^2}\Vert\nabla(\partial_X v)\Vert_{L^2}\\
&\leqslant C_1(\epsilon)\Vert \nu(x)\nabla u\Vert_{L^2}^{2} +C \epsilon \Vert\nabla(\partial_X v)\Vert_{L^2}^2 
.\end{split}
\end{equation}
The $5^{th}$ term of the LHS of the (\ref{eq:weak_eq}) we could rewrite in the following way 
\begin{equation}
V=-\int_{\Omega} X^{i}_{,k} \partial_{x_i}p \,X^{l}\partial_{x_l} v^k \,dx=2\left[\int_{\Omega}p\,X^{i}_{,k} \partial_{x_i}\partial_X v^k+\int_{\Omega}p\,X^{i}_{,ki}\partial_X v^k\right] 
.\end{equation}
Combining the above estimates of all terms of (\ref{eq:weak_eq}), we obtain 
\begin{equation}\label{eq:3.13_eq}
\begin{split}
\int_{\Omega} \nu (x)\mathbf{D} (\partial_X u) :\mathbf{D} (\partial_X v) \, dx &\leqslant \int_{\Omega}\vert \partial_{X} p\vert \vert X^i_k \partial_{x_i} v \vert\, dx+ \int _{\Omega} \nu(x)\left( \vert\nabla u\vert\vert \nabla X \vert\right) \vert\nabla(\partial_X v)\vert dx\\
&+\int_{\Omega}\vert \nu(x)\mathbf{D}u\vert \vert X^i_{,ki}\,\partial_X v+X^i_{,k}\,\nabla\partial_X v \vert \,dx\\
&+\int_{\Omega}\vert p\vert\vert X^i_{,ki}\,\partial_X v+X^i_{,k}\,\nabla\partial_X v \vert\, dx+\int_{\Omega}\vert \partial_{X}f \vert\vert\partial_X v\vert\,dx\\
&\leqslant C_{X'}\Vert \partial_{X} p\Vert_{L^2} \Vert\nabla v \Vert_{L^2}+c_{X'}c_{\nu}\Vert\nabla u\Vert_{L^2} \Vert \nabla \partial_X v\Vert_{L^2} \\
&+(c_p c_{X''}+c_{X'})\Vert\nu(x)\mathbf{D}u\Vert_{L^2} \Vert \nabla \partial_X v\Vert_{L^2}\\
&+(c_p c_{X''}+c_{X'})\Vert p\Vert_{L^2} \Vert\nabla\partial_{X} v\Vert_{L^2}+\Vert \partial_{X}f \Vert_{L^2} \Vert\partial_X v\Vert_{L^2}
.\end{split}
\end{equation}
We take $u=v$ in the above expression. Using the \eqref{eq:bogov2_eq} and Young's inequality with small $\epsilon$ we get
\begin{equation}
\begin{split}
\int_{\Omega} \nu (x)\vert\mathbf{D} (\partial_X u) \vert^2 dx & \leqslant C_{b_1}\epsilon\Big(\Vert\nu (x)\mathbf{D}\partial_X u\Vert_{L^2}^2 +c_2\Vert\nabla u \Vert_{L^2}^2 \Big) \\
&+c_3\left(\Vert \nu(x)\mathbf{D}u\Vert_{L^2}^2 + \Vert p\Vert_{L^2}^2\right)+c_p \Vert\partial_X f\Vert_{L^2}^2 )+C(\epsilon)\Vert \nabla u\Vert_{L^2}^2\\
&+C(\epsilon)[ c_2 \Vert\nabla u\Vert_{L^2}^2+c_3 (\Vert\nu(x)\mathbf{D}u\Vert_{L^2}^2+ \Vert p\Vert_{L^2}^2)+c_p\Vert \partial_{X}f \Vert_{L^2}^2 ]+\epsilon\Vert\nabla\partial_X u\Vert_{L^2}^2
\end{split}
\end{equation}
where $(c_p c_{X''}+c_{X'}):=c_3,\,  c_{X'}c_{\nu}=:c_2 .$

We apply the basic energy estimates (\ref{eq:bogov1}) and (\ref{2.3}) to get 
\begin{equation}
\begin{split}
\int_{\Omega} \nu (x)\vert\mathbf{D} (\partial_X u) \vert^2 dx 
&\leqslant C_{b_1}\epsilon\Vert\nu (x)\mathbf{D}\partial_X u\Vert_{L^2}^2  +(C_{b_1}\epsilon+C(\epsilon))\Big(c_{2}\Vert\nabla u \Vert_{L^2}^2 \\
&+c_4\left(\Vert \nabla u\Vert_{L^2}^2 + \Vert f\Vert_{L^2}^2\right)+c_p \Vert\partial_X f\Vert_{L^2}^2 \Big)+C(\epsilon)\Vert \nabla u\Vert_{L^2}^2+\epsilon\Vert\nabla\partial_X u\Vert_{L^2}^2\\
&\leqslant C_{b_1}\epsilon\Vert\nu (x)\mathbf{D}\partial_X u\Vert_{L^2}^2 +c_5 \Vert f\Vert_{L^2}^2 +c_6 \Vert\partial_X f\Vert_{L^2}^2+\epsilon\Vert\nabla\partial_X u\Vert_{L^2}^2
.\end{split}
\end{equation}
For sufficiently small $\epsilon$ we obtain
\begin{equation}
\begin{split}\label{3.21}
\int_{\Omega} \nu(x)\vert\mathbf{D}(\partial_X u) \vert^2 dx & \leqslant A \Big(c_5 \Vert f\Vert_{L^2}^2 +c_6 \Vert\partial_X f\Vert_{L^2}^2+\epsilon\Vert\nabla\partial_X u\Vert_{L^2}^2\Big)
,\end{split}
\end{equation} 
where $A=1/(1-C_{b_1}\epsilon)$.

From (\ref{eq:korn2_eq}) and (\ref{3.21}) follows  

\begin{equation}\label{3.22}
\begin{split}
 C\Vert\nabla\partial_X u\Vert_{L^2}^2 &\leqslant \int_{\Omega} \nu(x)\vert\mathbf{D}(\partial_X u) \vert^2 dx \\
&\leqslant A (c_5 \Vert f\Vert_{L^2}^2 +c_6 \Vert\partial_X f\Vert_{L^2}^2+\epsilon\Vert\nabla\partial_X u\Vert_{L^2}^2)
.\end{split}
\end{equation} 
The last term on the RHS we can put on the LHS, and for small enough $\epsilon$ we obtain desired inequality (\ref{eq:tan_regSt}).

\end{proof}

\subsection{The Model Case.} 
In this section, we assume that $\Omega$ is whole $\mathbb{R}^2$. In such a way that interior of the domain $S$ is half-space $x_2 <0$, and exterior is $x_2 >0$, i.e.
\[S:=\lbrace x\in \mathbb{R}^2\, \vert \, x_2 < 0 \rbrace.\]
In this case the derivative along the tangential vector field takes form  
\[\partial_{X}=\partial_{x_1}.\]
Obviously, viscosity (\ref{visc}) becomes a jump function along direction $x_2$ :
\begin{equation}
\nu(x_2)= \begin{cases}
1, & x_2 >0\\
m, & x_2 < 0
.\end{cases}
\end{equation}
One of the difficulties in proving higher regularity in the case with $n=p=2$ is we do not have embedding $H^1\nrightarrow L^{\infty}$. On the other hand, the jump function $\nu (x)$ does not belong to space $H^1$.
Thus, we work in the spaces defined in (\ref{1-dimsp})-(\ref{1-Dsp}), reducing our problem to space dimension one, where we have embedding $H^1\hookrightarrow L^{\infty}$. Here we just show a formal estimate can be obtained for this particular case.
This subsection gives us the crucial idea of proof of Theorem \ref{thm_main} and  the following lemma reads 
 
 \begin{lemma}\label{lem_ModelCase} Assume that domain $ \Omega =\mathbb{R}^2$ and let $\partial_{x_1}^k {f}\in L^2(\mathbb{R}^2)$ for some $k\in \mathbb{N}$. Then for every solution $u$ of Stokes equations (\ref{stokes_eq})-(\ref{div_eq}) we have $\partial_{x_1} u\in H^1(\mathbb{R}_{x_2},H^k(\mathbb{R}_{x_1}))$. 
Moreover, $\nabla u\in L^{\infty}(\mathbb{R}^2)$.
  
\end{lemma}
\begin{proof}
Let us propagate the Stokes equation over the given vector field

\begin{equation}\label{3.28model}
 -{\rm div\,}\left[\nu(x)\mathbf{D}\partial_{x_1} u\right] +\nabla{\partial_{x_1}p}=\partial_{x_1}f
 .\end{equation}
 Multiplying by the test function $\partial_{x_1} v\in H^1_0 (\mathbb{R}^2)$ and integrating by parts we get 
 \[\int_{\mathbb{R}^2} \nu(x_2) \vert \mathbf{D}\partial_{x_1} u\vert \vert \mathbf{D}\partial_{x_1} v\vert\leqslant C \Vert \partial_{x_1} f\Vert _2\Vert\partial_{x_1}v\Vert_{L^2} .\]
Testing by function $\partial_{x_1} u$, we could bound from below with Korn inequality
 \[c \Vert \nabla \partial_{x_1}u\Vert_{L^2}^2\leqslant\int_{\mathbb{R}^2}\nu(x_2) \vert \mathbf{D}\partial_{x_1} u\vert^2 , \] 
 and we have
\[\Vert \nabla \partial_{x_1}u\Vert_{L^2}\leqslant c_1 \Vert\partial_{x_1} f\Vert _2 . \]
So, we get that $\partial_{x_1}u\in H^1(\mathbb{R}^2)$.
 
Now, we will follow the same procedure as above the for second time and assume $\partial^2_{x_1}{f}\in L^2(\mathbb{R}^2)$. Differentiating (\ref{3.28model}) in $x_1$, we get

\begin{equation}
 -{\rm div\,}\left[\nu(x)\mathbf{D}\partial^2_{x_1} u\right] +\nabla{\partial^2_{x_1}p}=\partial^2_{x_1}f
. \end{equation}
 We have
 \[\Vert \nabla \partial^2_{x_1}u\Vert_{L^2}\leqslant c_2 \Vert\partial^2_{x_1} f\Vert _2 .\]

 From the above estimates we deduce that
 \begin{equation}
u^{2}_{,1} \in H^1 (\mathbb{R}_{x_2};H^2(\mathbb{R}_{x_1})),\,  u^{1}_{,1} \in H^1 (\mathbb{R}_{x_2};H^2(\mathbb{R}_{x_1}))
 .\end{equation}
If $\partial^{k}_{x_1}f\in L^2 (\mathbb{R}^2)$, by iteration of the same procedure as above it follows that
\begin{equation}
\nabla \partial_{x_1}^k u\in L^2(\mathbb{R}^2)
\end{equation}
for some $k\in \mathbb{N}$. From standard theory [\ref{temam}] (Ch.I, Proposition 2.2.), for this weak solution $u$ we can deduce an existence of pressure $p $, which is high regular in the $x_1$ direction.

Let us rewrite the first row of Stokes equation in the following way
\[-\nu(x_2) u_{,11}^1-\frac{1}{2}\,\partial_{x_2}\left[\nu(x)(u_{,2}^{1}+u_{,1}^{2})\right]+ \partial_{x_1}p = f^1. \]
Because differentiation in $x_1$ direction over $\nu(x_2)$ is well defined, we transfer this term to the RHS. The fact that the weak derivative of $p$ is in $L^2$, can be proved by standard techniques like difference quotients in Evans [\ref{evans}].
\begin{equation}\label{St1row}
\partial_{x_2}\left[\nu(x_2)(u_{,2}^{1}+u_{,1}^{2})\right]=2 f^1-2 p_{x_1}+2\nu(x_2)u^1_{,11} .
\end{equation}
Taking the $L^2$ norm in the direction $x_1$ we get
\begin{equation}\label{3.28}
\int _{\mathbb{R}_{x_1}}\left| \partial_{x_2}\left[\nu(x_2)(u_{,2}^{1}+u_{,1}^{2}) \right]\right|^2\, dx_1 \leqslant C \Big( \int_{\mathbb{R}_{x_1}} \vert f^1 \vert^2 \, dx_1+\int_{\mathbb{R}_{x_1}}\vert p_{x_1}\vert^2\,  dx_1+\int_{\mathbb{R}_{x_1}}\vert\nu(x_2)\partial_{x_1}^2 u^1 \vert^2\,  dx_1\Big).
\end{equation}
Now we differentiate (\ref{St1row}) by $x_1$ and take the $L^2(\mathbb{R}_{x_2};L^2(\mathbb{R}_{x_1}))$ norm:
\begin{equation}\label{3.31}
\begin{split}
\Big\Vert \partial_{x_2}\left[\nu(x)(u_{,21}^{1}+u_{,11}^{2})\right]\Big\Vert_{L^2(\mathbb{R}_{x_2};L^2(\mathbb{R}_{x_1}))}&\leqslant C_1 \Big( \Vert\partial_{x_1} f^1\Vert_{{L^2 (\mathbb{R}_{x_2}};{L^2(\mathbb{R}_{x_1})})}+\Vert \partial_{x_1}p_{x_1}\Vert_{{L^2 (\mathbb{R}_{x_2}};{L^2(\mathbb{R}_{x_1})})}\\
&+\Vert\nu(x_2)\partial_{x_1}^3 u \Vert_{{L^2 (\mathbb{R}_{x_2}};{L^2(\mathbb{R}_{x_1})})}\Big)
.\end{split}
\end{equation}

From (\ref{3.31}), by differentiating in $x_1$ direction, we have $$\partial_{x_2}\left[\nu(x)(u_{,2}^{1}+u_{,1}^{2} ) \right]\in L^{2} (\mathbb{R}_{x_2};H^1 (\mathbb{R}_{x_1}))\hookrightarrow L^{2}(\mathbb{R}_{x_2};L^{\infty} (\mathbb{R}_{x_1})),$$  
which implies that $\left[\nu(x)(u_{,2}^{1}+u_{,1}^{2})\right] \in H^1 (\mathbb{R}_{x_2};H^1(\mathbb{R}_{x_1}))\subset L^{\infty}.$ \\
From (\ref{div_eq}) we have $u^1_{,1}=-u^2_{,2}\in H^1$ globally.
It follows that $u \in H^1 (\mathbb{R}_{x_2}; H^1(\mathbb{R}_{x_1}))$.\\


From the above considerations we have an estimate
\begin{equation}\label{3.32}
\begin{split}
\Vert\nu(x_2)(u_{,2}^{1}+u_{,1}^{2})\Vert_{{H^1 (\mathbb{R}_{x_2}};{H^1 (\mathbb{R}_{x_1})})}&\leqslant C_2\Big( \Vert f^1\Vert_{L^2 (\mathbb{R}_{x_2};H^1(\mathbb{R}_{x_1}))}+\Vert p_{x_1}\Vert_{L^2 (\mathbb{R}_{x_2};H^1(\mathbb{R}_{x_1}))}\\
&+\Vert\nu(x_2)\partial_{x_1}^2 u^1\Vert_{{L^{2} (\mathbb{R}_{x_2}};H^1(\mathbb{R}_{x_1}))}\Big)
.\end{split}
\end{equation}

Due to embedding $H^1(\mathbb{R})\hookrightarrow L^{\infty}(\mathbb{R})$, the following inequality holds
\begin{equation}\label{embed}
\Vert u\Vert_{L^{\infty}(\mathbb{R})}\leqslant C \Vert u\Vert_{H^1(\mathbb{R})}
.\end{equation}

Using this, we could estimate LHS of (\ref{3.32}) from below
\begin{equation}\label{3.34}
\begin{split}
  C \Vert\nu(x_2)(u_{,2}^{1}+u_{,1}^{2})\Vert_{{L^{\infty} (\mathbb{R}_{x_2}};{H^1 (\mathbb{R}_{x_1})})}&\leqslant \Vert\nu(x_2)(u_{,2}^{1}+u_{,1}^{2})\Vert_{{H^1 (\mathbb{R}_{x_2}};{H^1 (\mathbb{R}_{x_1})})}\\
 &\leqslant C_2 \Big(\Vert f^1\Vert_{L^2 (\mathbb{R}_{x_2};H^1(\mathbb{R}_{x_1}))}+\Vert p_{x_1}\Vert_{L^2 (\mathbb{R}_{x_2};H^1(\mathbb{R}_{x_1}))}\\
 &+\Vert\nu(x_2)\partial_{x_1}^2 u^1\Vert_{{L^{2} (\mathbb{R}_{x_2}};H^1(\mathbb{R}_{x_1}))}\Big)
.\end{split}
\end{equation}

We know that $u^2_{,1}\in H^1(\mathbb{R}_{x_2},H^k(\mathbb{R}_{x_1}))$, so we can bound $u^1_{,2}$ using triangle inequality
\begin{equation}
  \Vert\nu(x_2)u_{,2}^{1}\Vert_{{L^{\infty} (\mathbb{R}_{x_2}};{H^1 (\mathbb{R}_{x_1})})}\leqslant \Vert\nu(x_2)(u_{,2}^{1}+u_{,1}^{2})\Vert_{L^{\infty} (\mathbb{R}_{x_2};{H^1 (\mathbb{R}_{x_1})})}+\Vert\nu(x_2)u_{,1}^{2}\Vert_{L^{\infty} (\mathbb{R}_{x_2};H^1 (\mathbb{R}_{x_1}))}
.\end{equation}
We could use  the following inequality 
\begin{equation}
\Vert u_{,2}^{1}\Vert_{L^{\infty} (\mathbb{R}_{x_2};L^{\infty} (\mathbb{R}_{x_1}))} \leqslant \Big\Vert\dfrac{1}{\nu(x_2)}\Big\Vert_{L^{\infty}(\mathbb{R}_{x_2})}\Vert \nu(x_2) u_{,2}^{1}\Vert_{L^{\infty} (\mathbb{R}_{x_2};L^{\infty} (\mathbb{R}_{x_1}))}
.\end{equation}
Using triangle inequality and the above inequality we get
\begin{equation}
\begin{split}
 \Vert u_{,2}^{1}\Vert_{L^{\infty} (\mathbb{R}_{x_2};L^{\infty} (\mathbb{R}_{x_1})))} &\leqslant \Big\Vert\dfrac{1}{\nu(x_2)}\Big\Vert_{L^{\infty}(\mathbb{R}_{x_2})}\lbrace\Vert\nu(x_2)(u_{,2}^{1}+u_{,1}^{2})\Vert_{L^{\infty} (\mathbb{R}_{x_2};{H^1 (\mathbb{R}_{x_1})})}\\
 &+\Vert\nu(x_2)u_{,1}^{2}\Vert_{L^{\infty} (\mathbb{R}_{x_2};H^1 (\mathbb{R}_{x_1}))}\rbrace\\
&\leqslant C_3 \Big\Vert\dfrac{1}{\nu(x_2)}\Big\Vert_{L^{\infty}(\mathbb{R}_{x_2})}\lbrace \Vert f^1\Vert_{L^2 (\mathbb{R}_{x_2};H^1(\mathbb{R}_{x_1}))}+\Vert p_{x_1}\Vert_{L^2 (\mathbb{R}_{x_2};H^1(\mathbb{R}_{x_1}))}\\
&+c\Vert\nu(x_2)\partial_{x_1}^2 u^1\Vert_{{L^{2} (\mathbb{R}_{x_2}};H^1(\mathbb{R}_{x_1}))}+\Vert\nu(x_2)u_{,1}^{2}\Vert_{L^{\infty} (\mathbb{R}_{x_2};H^1 (\mathbb{R}_{x_1}))}\rbrace
.\end{split}
\end{equation}
It follows that $u^2_1\in L^{\infty} (\mathbb{R}_{x_2};L^{\infty} (\mathbb{R}_{x_1}))) $.
Now, consider the second row of Stokes equation:
\[-\partial_{x_2} (\nu(x)u^2_{,2})=\frac{1}{2}\,\partial_{x_1}\left[\nu(x)(u^1_{,2}+u^2_{,1}) \right]+f^2+\partial_{x_2}p .\]

We take the $L^2(\mathbb{R}_{x_2};L^2(\mathbb{R}_{x_1}))$ norm and differentiate (\ref{St1row}) by $x_1$ the above expression,
\begin{equation}\label{3.37}
\begin{split}
\Big\Vert\partial_{x_1}\partial_{x_2}\left[ \nu(x_2 )u^2_{,2}\right] \Big\Vert_{L^2 (\mathbb{R}_{x_2 };L^2  (\mathbb{R}_{x_1}))}&\leqslant C \Big(\Vert\nu(x_2 )(u^1_{,2}+u^2_{,1})\Vert_{H^1(\mathbb{R}_{x_2};H^2 (\mathbb{R}_{x_1}))}\\
&+\Vert \partial_{x_2} p\Vert_{L^2(\mathbb{R}_{x_2};H^1 (\mathbb{R}_{x_1}))}+\Vert f\Vert_{L^2 (\mathbb{R}_{x_2};H^1 (\mathbb{R}_{x_1}))}\Big)
.\end{split}
\end{equation} 
Because we know that differentiating in $x_1$ direction is regular. We know that $u^2_{,2}\in H^1 (\mathbb{R}_{x_2};H^1(\mathbb{R}_{x_1}))$, so we could improve the regularity of the above inequality
\begin{equation}
\begin{split}
\Vert \nu(x_2 )u^2_{,2} \Vert_{H^1 (\mathbb{R}_{x_2 };H^1  (\mathbb{R}_{x_1}))}&\leqslant C \Big(\Vert\nu(x_2 )(u^1_{,2}+u^2_{,1})\Vert_{H^1 (\mathbb{R}_{x_2};H^2 (\mathbb{R}_{x_1}))}\\
&+\Vert \partial_{x_2}p\Vert_{L^2 (\mathbb{R}_{x_2};H^1 (\mathbb{R}_{x_1}))}+\Vert f\Vert_{L^2 (\mathbb{R}_{x_2};H^1 (\mathbb{R}_{x_1}))}\Big)
\end{split}
\end{equation}
and using embedding (\ref{embed}), we get
\begin{equation}
C \Vert \nu(x_2 )u^2_{,2} \Vert_{L^{\infty} (\mathbb{R}_{x_2 };H^1  (\mathbb{R}_{x_1}))}\leqslant
\Vert \nu(x_2 )u^2_{,2} \Vert_{H^1 (\mathbb{R}_{x_2 };H^1  (\mathbb{R}_{x_1}))}
. \end{equation}
 As in the previous case we obtain
 \begin{equation}
\begin{split}
\Vert u^2_{,2} \Vert_{L^{\infty} (\mathbb{R}_{x_2 };L^{\infty}  (\mathbb{R}_{x_1}))}&\leqslant C_1 \Big\Vert\dfrac{1}{\nu(x_2)}\Big\Vert_{L^{\infty}(\mathbb{R}_{x_2})}\lbrace  \Vert\nu(x_2 )(u^1_{,2}+u^2_{,1})\Vert_{H^1(\mathbb{R}_{x_2};H^2 (\mathbb{R}_{x_1}))}\\
&+\Vert \partial_{x_2}p\Vert_{L^2 (\mathbb{R}_{x_2};H^1 (\mathbb{R}_{x_1}))}+\Vert f\Vert_{L^2 (\mathbb{R}_{x_2};H^1 (\mathbb{R}_{x_1}))}\rbrace
.\end{split}
\end{equation}
 It implies that $u^2_2\in L^{\infty} (\mathbb{R}_{x_2};L^{\infty} (\mathbb{R}_{x_1}))) $. From estimates of $u$ derivatives we conclude that
  $$\nabla u\in  L^{\infty} (\mathbb{R}^2).$$
\end{proof}

\subsection{ The second derivative.}
In this subsection, we introduce a lemma that gives the higher tangential regularity of the solution to Stokes problem. We first study the tangential regularity of the Stokes system, which will be useful in the proof of Lemma \ref{higtan}. The result states
\begin{lemma}\label{l_hihgerRegSt} Let $\Omega \subset \mathbb{R}^2 $ be an open bounded domain of class $C^2$, $S\subset \Omega$ with the boundary $\partial{S}\in C^2 $, and
let $X$ be a vector field satisfying (\ref{X}). Assume 
$f, \partial_X f, \partial^2_X f\in L^2 (\Omega)$, and let $u$ solve (\ref{stokes_eq})-(\ref{div_eq}).
Then  $\partial^2_{X} u\in H^1(\Omega)$, with an estimate
\begin{equation}\label{ineq:gencase}
\Vert\nabla \partial^2_X u\Vert_{L^2}\leqslant C_2\left(  \Vert f\Vert_{L^2}+ \Vert \partial_X f\Vert_{L^2}+\Vert \partial^2_X f\Vert_{L^2}\right) 
\end{equation}   
where $C_2 :=C_2 (\Omega ,X)$.
\end{lemma}
\begin{proof}
Let us take the second derivative along tangential vector field from the Stokes equation, i.e. we take derivative $\partial_{X}$ from (\ref{eq:3.7_eq}).
\begin{equation}
\begin{split}\label{eq:gen2_eq}
-{\rm div\,}\big[&\nu (x) \mathbf{D}(\partial^2_X u) \big] 
+\frac{1}{2}{\rm div\,}(\nu(x)\nabla^T X \nabla^T \partial_X u+\nu(x)\nabla \partial_X u \nabla X)+\sum_{i}X^i_{,k}\partial_{x_i}\left( \nu(x)\mathbf{D}(\partial_X u)\right) \\
+&\partial_X\left[ \frac{1}{2}{\rm div\,}(\nu(x)\nabla^T X \nabla^T u+\nu(x)\nabla u \nabla X)+\sum_{i}X^i_{,k}\partial_{x_i}\left( \nu(x)\mathbf{D}u\right)\right] \\
+&\partial_X \left[ \nabla (\partial_X p)-\sum_{i}X^i_{,k}\partial_{x_i} p\right] =\partial^2_X  f
.\end{split}
\end{equation}
We will rewrite the above expression term by term. 
Straightforward calculations of the $4^{th}$ term of the LHS of (\ref{eq:gen2_eq}) gives
\begin{equation}
\begin{split}
\partial_X\big[& \frac{1}{2}{\rm div\,}(\nu(x)\nabla^T X \nabla^T u+\nu(x)\nabla u \nabla X)\big]= \frac{1}{2}{\rm div\,}(\nu(x)\partial_X (\nabla^T X \nabla^T u+\nabla u \nabla X))\\
&- \frac{1}{2} X^s_{,k}\partial_{x_s}(\nu(x)(\nabla^T X \nabla^T u+\nabla u \nabla X))= \frac{1}{2}{\rm div\,}\lbrace\nu(x) (\nabla^T X \nabla^T \partial_X u+\nabla\partial_X u \nabla X)\\
+&\nu(x) (X^s X^j_{,is}\, \nabla^T u+\nabla u\, X^s X^i_{,js})-\nu(x) ( X^j_{,i}X^s_{,i}\, \nabla^T u+\nabla u \,X^i_{,j}X^s_{,j})\rbrace\\
-& \frac{1}{2} X^s_{,k}\partial_{x_s}(\nu(x)(\nabla^T X \nabla^T u+\nabla u \nabla X))
.\end{split}
\end{equation}
For the $5^{th}$ term we have
\begin{equation}
\begin{split}\label{eq:5t_eq}
\partial_X\left[X^i_{,k}\partial_{x_i}\left( \nu(x)\mathbf{D}u\right)\right] & = X^s X^i_{,ks}\partial_{x_i}\left( \nu(x)\mathbf{D}u\right) +X^i_{,k}\partial_{x_i}\left( \nu(x)\partial_X\mathbf{D}u\right)-X^i_{,k}X^s_{,i}\partial_{x_s}\left( \nu(x)\mathbf{D}u\right)\\
= X^s & X^i_{,ks}\partial_{x_i}\left( \nu(x)\mathbf{D}u\right)-X^i_{,k}X^s_{,i}\partial_{x_s}\left( \nu(x)\mathbf{D}u\right)\\
&+X^i_{,k}\left( \nu(x)\mathbf{D}\partial_X u\right)-X^i_{,k}\partial_{x_i}(\nu(x)(\nabla^T X \nabla^T u+\nabla u \nabla X))
.\end{split}
\end{equation}

Also, the $6^{th}$ term with pressure
\begin{equation}
\partial_X (\nabla \partial_X p)-X^s\partial_{x_s}(X^i_{,k}\partial_{x_i} p)=\nabla \partial_X^2 p-X^s_{,k}\partial_{x_s} \partial_X p-X^s X^i_{,ks}\partial_{x_i} p- X^i_{,k}\partial_{x_i}\partial_X p+X^s_{,i}X^i_{,k}\partial_{x_s} p
.\end{equation}
Multiplying the equation (\ref{eq:gen2_eq}) by $\psi\in C^{\infty}_0(\Omega) $  and integrating, we get
 
\begin{equation}
\begin{split}\label{eq:g_eq}
(\nu (x)& \mathbf{D}(\partial^2_X u),\mathbf{D}(\psi)) +( \nu(x)(\nabla^T X \nabla^T \partial_X u+\nabla \partial_X u \nabla X),\psi)\\
&+\frac{1}{2}(\nu(x) (X^s X^j_{,is}\, \nabla^T u+\nabla u\, X^s X^i_{,js}),\psi)-\frac{1}{2}(\nu(x) ( X^j_{,i}X^s_{,i}\, \nabla^T u+\nabla u \,X^i_{,j}X^s_{,j}),\psi)\\
&-( \nu(x)\mathbf{D}(\partial_X u), \partial_{x_i}(X^i_{,k}\psi))+\frac{3}{2}(\nu(x) (\nabla^T X\nabla^T u+\nabla u \nabla X),\partial_{x_s}(X^s_{,k}\psi))\\
&-(\partial_X^2 p,{\rm div\,}\psi)+(\partial_X p, \partial_{x_s}(X^s_{,k}\psi))+(p, \partial_{x_i}(X^s X^i_{,ks}\psi))\\
&+(\partial_X p, \partial_{x_i}(X^i_{,k}\psi))-(p, \partial_{x_i}(X^s_{,i} X^i_{,k}\psi))=(\partial^2_X f,\psi)
.\end{split}
\end{equation}
The weak formulation of \eqref{eq:gen2_eq}, tested with $\partial^2_{X}v\in H^{1}_{0} (\Omega)$ gives
\begin{equation}
\begin{split}\label{eq:gweak3_eq}
\int_{\Omega}&\nu (x) \mathbf{D}(\partial^2_X u) \mathbf{D}(\partial^2_X v)\,dx + \int_{\Omega}\lbrace \nu(x)(\nabla^T X \nabla^T \partial_X u+\nabla \partial_X u \nabla X)\\
&+\frac{1}{2}\nu(x) (X^s X^j_{,is}\, \nabla^T u+\nabla u\, X^s X^i_{,js})-\frac{1}{2}\nu(x) ( X^j_{,i}X^s_{,i}\, \nabla^T u+\nabla u \,X^i_{,j}X^s_{,j})\rbrace:\partial^2_X v\, dx\\
&+\int_{\Omega}X^i_{,k}\partial_{x_i}\left( \nu(x)\mathbf{D}(\partial_X u)\right) \partial^2_X v\, dx-\frac{1}{2}\int_{\Omega}X^s_{,k}\partial_{x_s}(\nu(x) (\nabla^T X\nabla^T u+\nabla u \nabla X))\partial^2_X v\, dx\\
&-\int_{\Omega}X^i_{,k}\partial_{x_i}(\nu(x) (\nabla^T X\nabla^T u+\nabla u \nabla X))\partial^2_X v\, dx\\
&-\int \partial_X^2 p\,{\rm div\,}\partial^2_X v\, dx+\int_{\Omega} \partial_X p \partial_{x_s}(X^s_{,k}\partial^2_X v)\, dx+\int_{\Omega} p\, \partial_{x_i}(X^s X^i_{,ks}\partial^2_X v)\, dx\\
&+\int_{\Omega}\partial_X p\, \partial_{x_i}(X^i_{,k}\partial^2_X v)\, dx-\int_{\Omega} p\, \partial_{x_i}(X^s_{,i} X^i_{,k}\partial^2_X v)\, dx=\int_{\Omega}\partial^2_X f\partial^2_X v\, dx
.\end{split}
\end{equation}
\textbf{The pressure term: Bogovskii type estimate}.  We are going to estimate the pressure term $\partial_X^2 p $ in the view of Lemma \ref{lem_bogovskii}. We need to show that $\partial_X^2 p\in L^2(\Omega) $ such that equation \eqref{eq:gweak3_eq} holds for every $\psi\in C^{\infty}_0(\Omega) $. Moreover, the pressure term is well defined 
\[ \left| \int_{\Omega} \partial_{X}^2 p\, dx\right| =\left| -\int_{\Omega} {\rm div\,}X \partial_{X} p \,dx +\int_{\partial\Omega} X^i n^i \partial_{X} p\, ds\right| \leqslant \Vert{\rm div\,}X\Vert_{L^2} \Vert \partial_{X} p\Vert_{L^2} .\;\;\;\, (**)\]
We've got the above inequality by knowing that $\partial_{X}p$ is bounded in $L^2(\Omega)$ according to the proof of Lemma \ref{lem_trSt}, and also using integration by parts and H\"older inequality.

Let us consider the functional

\begin{equation}
\begin{split}
\mathcal{F}(\psi)& =(\nu (x) \mathbf{D}(\partial^2_X u),\mathbf{D}(\psi)) +( \nu(x)(\nabla^T X \nabla^T \partial_X u+\nabla \partial_X u \nabla X),\psi)\\
&+\frac{1}{2}(\nu(x) (X^s X^j_{,is}\, \nabla^T u+\nabla u\, X^s X^i_{,js}),\psi)-\frac{1}{2}(\nu(x) ( X^j_{,i}X^s_{,i}\, \nabla^T u+\nabla u \,X^i_{,j}X^s_{,j}),\psi)\\
&-( \nu(x)\mathbf{D}(\partial_X u), \partial_{x_i}(X^i_{,k}\psi))+\frac{3}{2}(\nu(x) (\nabla^T X\nabla^T u+\nabla u \nabla X),\partial_{x_s}(X^s_{,k}\psi))\\
&+2(\partial_X p, \partial_{x_s}(X^s_{,k}\psi))+(p, \partial_{x_i}(X^s X^i_{,ks}\psi))\\
&-(p, \partial_{x_i}(X^s_{,i} X^i_{,k}\psi))-(\partial^2_X f,\psi)
 .\end{split}
\end{equation}
For all $\psi\in H^{1}_0 (\Omega)$. 
Thinking on the level of the weak formulation of Stokes system with a test function $(\partial_{X}^2)^{*}\psi \in C^{\infty}_0(\Omega)$ we get
 \[(p,{\rm div\,}(\partial_{X}^2)^{*}\psi)=\langle\partial_X^2 \nabla p, \psi \rangle  \]
 where $(\partial_{X}^2)^{*}\psi:= \partial_{x_k}(X^{k}\, \partial_{x_i}(X^{i}\, \psi))$. 
We deduce by Lemma \ref{lem_bogovskii} there exists a uniquely determined $\partial_{X}^2 p\in L^2 (\Omega)$ with bounded
$\frac{1}{\vert\Omega\vert}\int _{\Omega}\partial_{X}^2 p$, 
such that 
\begin{equation}\label{fcl_2HR}
\mathcal{F}(\psi)=(\partial_{X}^2 p,{\rm div\,}\psi)
\end{equation}
for all $\psi\in H^{1}_0 (\Omega)$.
Consider the problem
\[{\rm div\,}\psi= \partial_{X}^2 p-\frac{1}{\vert\Omega\vert}\int _{\Omega}\partial_{X}^2 p =g.\]
\begin{equation}\label{eq:genpr_eq}
\psi \in H^{1}_0 (\Omega)
\end{equation}
\[\Vert\psi\Vert_{H^{1}}\leqslant C_{b_2} \Vert \partial_{X}^2 p\Vert_{L^2} \]
with $\Omega$ bounded and satisfying the cone condition. Since
$\frac{1}{\vert \Omega\vert}\int_{\Omega}\partial_{X}^2 p $ is bounded 
and
$$\int_{\Omega}g=0 ,\, g\in L^2(\Omega),\;\,\,\Vert\psi\Vert_{W^{1,2}}\leqslant C_b \Vert \partial_{X}^2 p\Vert_{L^2}$$ 
from Theorem III.3.1 ([\ref{galdi_book}]) we deduce the existence of $\psi$ solving the \eqref{eq:genpr_eq}, using such a $\psi$ as test function into the \eqref{eq:g_eq}, we have

\begin{equation}
\begin{split}
\Vert\partial_{X}^2 p\Vert_{L^2}^2 & =(\nu (x) \mathbf{D}(\partial^2_X u),\mathbf{D}(\psi)) +( \nu(x)(\nabla^T X \nabla^T \partial_X u+\nabla \partial_X u \nabla X),\psi)\\
&+\frac{1}{2}(\nu(x) (X^s X^j_{,is}\, \nabla^T u+\nabla u\, X^s X^i_{,js}),\psi)-\frac{1}{2}(\nu(x) ( X^j_{,i}X^s_{,i}\, \nabla^T u+\nabla u \,X^i_{,j}X^s_{,j}),\psi)\\
&-( \nu(x)\mathbf{D}(\partial_X u), \partial_{x_i}(X^i_{,k}\psi))+\frac{3}{2}(\nu(x) (\nabla^T X\nabla^T u+\nabla u \nabla X),\partial_{x_s}(X^s_{,k}\psi))\\
&+2(\partial_X p, \partial_{x_s}(X^s_{,k}\psi))+(p,X^s \partial_{x_i}( X^i_{,ks}\psi))\\
&-(p, X^i_{,k} \partial_{x_i}(X^s_{,i} \psi))-(\partial^2_X f,\psi)
.\end{split}
\end{equation}
By applying H\"older, Poincar\'e inequalities to the above equation, we get


\begin{equation}
\begin{split}
\Vert\partial_{X}^2 p\Vert_{L^2}^2 &\leqslant \lbrace \Vert\nu (x) \mathbf{D}(\partial^2_X u)\Vert_{L^2}  +c_p \Vert \nu(x)(\nabla^T X \nabla^T \partial_X u+\nabla \partial_X u \nabla X)\Vert_{L^2} \\
&+\frac{1}{2}c_p \Vert\nu(x) (X^s X^j_{,is}\, \nabla^T u+\nabla u\, X^s X^i_{,js})\Vert_{L^2} +\frac{1}{2}c_p \Vert\nu(x) ( X^j_{,i}X^s_{,i}\, \nabla^T u+\nabla u \,X^i_{,j}X^s_{,j})\Vert_{L^2} \\
&+(c_{X"}c_p +c_{X'})\Vert \nu(x)\mathbf{D}(\partial_X u)\Vert_{L^2} +\frac{3}{2}(c_{X"}c_p +c_{X'})\Vert\nu(x) (\nabla^T X\nabla^T u+\nabla u \nabla X)\Vert_{L^2} \\
&+2(c_{X"}c_p +c_{X'})\Vert\partial_X p\Vert_{L^2} +(c_{X}c_{X"}c_p +c_{X}c_{X"}+c_{X'}c_{X"}c_p+c_{X'}^2)\Vert p\Vert_{L^2} \\
&+c_{p}\Vert\partial^2_X f\Vert_{L^2} \rbrace\Vert\nabla\psi\Vert_{L^2}
 .\end{split}
\end{equation}
Then using the inequality from (\ref{eq:genpr_eq}), we could reduce both sides of the above inequality by $\Vert\partial_{X}^2 p\Vert$. Also, by applying (\ref{eq:tan_regSt}) and (\ref{2.3}) we deduce the estimate for the pressure term 
\begin{equation}\label{bogov3}
\Vert\partial_{X}^2 p\Vert_{L^2} \leqslant C'_{b_2}\lbrace \Vert\nu (x) \mathbf{D}(\partial^2_X u)\Vert_{L^2} + B_1\Vert f\Vert_{L^2}+F_1 \Vert \partial_X f \Vert_{L^2} +c_{p}\Vert\partial^2_X f\Vert_{L^2} \rbrace
 .\end{equation}

  Let us take the second directional derivative from the (\ref{div_X}) 
\begin{equation}
\begin{split} 
\partial_X^2  {\rm div\,} u &={\rm div\,}\partial_X^2 u - X^s_{,k}\partial_X u-X^s X^i_{,ks}\partial_{x_i} u- X^i_{,k} X^s \partial_{x_s} \partial_{x_i} u\\
&={\rm div\,}\partial_X^2 u - X^s_{,k}\partial_X u-X^s X^i_{,ks}\partial_{x_i} u-X^i_{,k} \partial_{x_i} \partial_{X} u+X^i_{,k} X^s_{,i} \partial_{x_s} u
\end{split} 
\end{equation}
so we get 
\begin{equation}\label{div2}
{\rm div\,}\partial_X^2 u = X^s_{,k}\partial_X u+X^s X^i_{,ks}\partial_{x_i} u+X^i_{,k} \partial_{x_i} \partial_{X} u-X^i_{,k} X^s_{,i} \partial_{x_s} u
 .\end{equation}
Now, we return to the \eqref{eq:gweak3_eq} with $u=v$, by applying H\"older inequality, we get
\begin{equation}
\begin{split}\label{eq:gwb_eq}
\int_{\Omega}&\vert\nu (x)\vert \mathbf{D}(\partial^2_X u)\vert^2\, dx
\leqslant \Vert\partial_{X}^2 p\Vert_{L^2}\Vert{\rm div\,}\partial^2_X u\Vert_{L^2}+\Vert \nu(x)(\nabla^T X \nabla^T \partial_X u+\nabla \partial_X u \nabla X)\Vert_{L^2} \Vert\partial^2_X u\Vert_{L^2}\\
&+\frac{1}{2}\Vert\nu(x) (X^s X^j_{,is}\, \nabla^T u+\nabla u\, X^s X^i_{,js})\Vert_{L^2} \Vert\partial^2_X u\Vert_{L^2}+\frac{1}{2}\Vert\nu(x) ( X^j_{,i}X^s_{,i}\, \nabla^T u+\nabla u \,X^i_{,j}X^s_{,j})\Vert_{L^2} \Vert\partial^2_X u\Vert_{L^2}\\
&+\Vert \nu(x)\mathbf{D}(\partial_X u)\Vert_{L^2} \Vert\partial_{x_i}(X^i_{,k}\partial^2_X u)\Vert_{L^2}+\frac{3}{2}\Vert\nu(x) (\nabla^T X\nabla^T u+\nabla u \nabla X)\Vert_{L^2} \Vert\partial_{x_s}(X^s_{,k}\partial^2_X u)\Vert_{L^2}\\
&+2\Vert\partial_X p\Vert_{L^2} \Vert\partial_{x_s}(X^s_{,k}\partial^2_X u)\Vert_{L^2}+\Vert p\Vert_{L^2} \Vert X^s \partial_{x_i}( X^i_{,ks}\partial^2_X u)\Vert_{L^2}\\
&+\Vert p\Vert_{L^2} \Vert X^i_{,k} \partial_{x_i}(X^s_{,i} \partial^2_X u)\Vert_{L^2}+\Vert\partial^2_X f\Vert_{L^2} \Vert\partial^2_X u\Vert_{L^2}
 .\end{split}
\end{equation}
 Here, using (\ref{div2}), (\ref{bogov3}) and applying Young's inequality with small $\epsilon$ to the above inequality, we get
\begin{equation}
\begin{split}
\int_{\Omega}\vert\nu (x) \mathbf{D}(\partial^2_X u) \vert^2\, dx
&\leqslant C'_{b_2 }\epsilon \Vert\nu (x) \mathbf{D}(\partial^2_X u)\Vert_{L^2}^2 \\
&+(C'_{b_2 }\epsilon +C(\epsilon))\lbrace B_1\Vert f\Vert_{L^2}^2 +F_1 \Vert \partial_X f \Vert_{L^2}^2 + \Vert\partial^2_X f\Vert_{L^2}^2\rbrace\\
& + C(\epsilon)\lbrace (c_{X'}c_p +c_{X'})(B_1\Vert f\Vert_{L^2}^2+F_1 \Vert \partial_X f \Vert_{L^2}^2) \rbrace\\
& +\epsilon\Vert\nabla\partial^2_X u\Vert_{L^2}^2
 .\end{split}
\end{equation}
Transferring the $1^{st}$ term of the RHS of the above inequality to the LHS we get
\begin{equation}\label{3.62}
\int_{\Omega}\vert\nu (x) \mathbf{D}(\partial^2_X u)\vert^2\, dx \leqslant  B_2\Vert f\Vert_{L^2}^2+F_2 \Vert \partial_X f \Vert_{L^2}^2 +A_{2}\Vert\partial^2_X f\Vert_{L^2}^2 +\epsilon\Vert\nabla\partial^2_X u\Vert_{L^2}^2
 .\end{equation}
 In the same way as in (\ref{bc}), we deduce that $\partial_{X}^2 u\vert_{\partial \Omega} =0$. Then Korn inequality holds 
 \begin{equation}\label{eq:korn3_eq}
\int_{\Omega} \nu (x)\vert\mathbf{D}(\partial_X^2 u)\vert^2 dx \geqslant C\int_{\Omega} \vert\nabla\partial_X^2 u\vert^2 dx
.\end{equation} 
 Implementing (\ref{eq:korn3_eq}) to (\ref{3.62}) gives required inequality (\ref{ineq:gencase}).

\end{proof}

\section{Proof of Theorem 2.3}

\subsection{Tangential Regularity.}
Tangential regularity result of approximate Navier-Stokes equations (\ref{NS}) reads

\begin{lemma}\label{l_PropNS} Let $\Omega \subset \mathbb{R}^2 $ be an open bounded domain of class $C^2$, $S\subset \Omega$ with the boundary $\partial{S}\in C^2 $,
and
let $X$ be a vector field satisfying (\ref{X}). Assume 
$f, \partial_X f \in L^2 (\Omega)$, and let $u$ solve (\ref{NS})-(\ref{NSdiv}). Then
 $\partial_{X} u\in H^1(\Omega)$, with an estimate
\begin{equation}\label{4.1}
 \Vert \nabla \partial_{X} u\Vert_{L^2}\leqslant C \left( \Vert f\Vert_{L^2} +\Vert \partial_{X} f\Vert_{L^2} +\Vert f\Vert_{L^2}^2  \right) 
\end{equation}
where $C :=C(\Omega ,X)$.

\end{lemma}
\begin{proof}
In this regard, we differentiate the stationary Navier-Stokes equation along the vector field 

\begin{equation}\label{PropNS}
\partial_{X}\left[ (u\cdot\nabla)u -{\rm div\,}\left[\nu(x)\mathbf{D} u\right] +\nabla{p}\right] =\partial_{X}f
\end{equation}
\begin{equation*}
\partial_{X} {\rm div\,}u=0
.\end{equation*}
We are going to show estimates for the nonlinear part of the equation, as the rest has been proved in Lemma \ref{lem_trSt}. 
\begin{equation}
\partial_{X}\left[(u\cdot\nabla)u\right]=(\partial_{X} u\cdot \nabla)u+( u\cdot \nabla)\partial_{X}u-u^i \partial_{x_i}X^l \partial_{x_l} u^k  
\end{equation}
let us separately test the nonlinear part by $\partial_{X}u$. We have
\begin{equation}\label{I}
I=\int_{\Omega} (\partial_{X} u\cdot \nabla)u\,\partial_{X} u \,dx \leqslant \Vert\partial_{X} u\Vert^2_{L^4}\Vert \nabla u\Vert_{L^2}
.\end{equation}
We use interpolation inequality for $q > p$ of type 
\[(L_p , BMO)_{1-\frac{p}{q},q}=L_{q}.\]
In our case $p=n=2,$ $q=4$
\begin{equation}\label{interpolation}
\Vert\partial_{X} u\Vert_{L^4}\leqslant C \Vert \partial_{X} u\Vert^{\frac{1}{2}}_{L^2}\Vert\partial_{X} u\Vert^{\frac{1}{2}}_{BMO}
,\end{equation}
where
\[\int\vert\partial_{X} u \vert^2 \, dx\leqslant \int \vert X\vert^2 \vert\nabla u\vert^2 \,dx\leqslant C \Vert \nabla u\Vert^2_{L^2} .\]
We implement (\ref{interpolation}) to (\ref{I}) and get that
\begin{equation}
I\leqslant C \Vert u\Vert_{H^1}\Vert \partial_{X} u\Vert_{L^2}\Vert\partial_{X} u\Vert_{BMO}\leqslant  C_1 \Vert u\Vert_{H^1}^2\Vert\partial_{X} u\Vert_{BMO}\leqslant 
C_1 \Vert u\Vert_{H^1}^2\Vert\partial_{X} u\Vert_{H^1}
.\end{equation}

From the property of trilinear form we have  $\int_{\Omega} b(v,u,u)\,dx=0$, so for the $II$ term we obtain

\begin{equation}
 II=\int_{\Omega}( u\cdot \nabla)\partial_{X}u\, \partial_{X}u\,dx=0
.\end{equation}

To the next term we apply general H\"older and Poincar\'e inequalities
\begin{equation}
III=\int_{\Omega}u^i \partial_{x_i}X^l \partial_{x_l} u^k X^s \partial_{x_s} u^k dx\leqslant \int_{\Omega}\vert u\nabla X \nabla u \partial_{X} u \vert dx\leqslant C \Vert u \Vert_{L^4}\Vert \nabla u\Vert_{L^2}\Vert\partial_{X} u\Vert_{L^4}
\end{equation}
i.e.
\begin{equation}
III\leqslant C_2 \Vert u \Vert_{H^1_0}^2\Vert\partial_{X} u\Vert_{H^1}
.\end{equation}
Summing up all the above estimates and using H\"older, Young' inequalities we get 
\begin{equation}\label{nonlinTR}
\Big\vert\int \partial_{X}\left[(u\cdot\nabla)u\right]\partial_{X} u\, dx \Big\vert \leqslant C_3 \Vert u \Vert_{H^1}^2\Vert\partial_{X} u\Vert_{H^1} \leqslant C_4 \Vert f \Vert_{L^2}^4+ C_5 \epsilon\Vert\partial_{X} u\Vert_{H^1}^2 
.\end{equation}
From Lemma IX 1.2 [\ref{galdi_book}] and the same way as in (\ref{3.10}) we deduce that there exists a uniquely determined $\partial_{X} p\in L^2 (\Omega)$ for Navier-Stokes system.  

The estimate (\ref{3.22}) for propagated Stokes equation combined with (\ref{PropNS}) give us an estimate 
\begin{equation}
C \Vert \nabla \partial_{X} u\Vert_{L^2}^2\leqslant C_6 \Vert f\Vert_{L^2}^2 +C_7 \Vert \partial_{X} f\Vert_{L^2}^2 +C_{4}\Vert f\Vert_{L^2}^4 +C_8 \epsilon \Vert \nabla \partial_{X} u \Vert_{L^2}^2   
.\end{equation}
Taking $\epsilon$ in such a way that $C >C_8 \epsilon $, we have (\ref{4.1}).

\end{proof}

The higher tangential regularity is an important part of the proof of Theorem \ref{thm_main}. It gives us higher regularity of solutions of Navier-Stokes problem (\ref{NS})-(\ref{NSdiv}). The following lemma states the result 
\begin{lemma}\label{higtan} Let $\Omega \subset \mathbb{R}^2 $ be an open bounded domain of class $C^2$, $S\subset \Omega$ with the boundary $\partial{S}\in C^2 $, and
let $X$ be a vector field satisfying (\ref{X}). Assume 
$f, \partial_X f, \partial^2_X f \in L^2 (\Omega)$, and let $u$ solve (\ref{NS})-(\ref{NSdiv}). Then  $\partial_{X}^2 u\in H^1(\Omega)$, with an estimate
\begin{equation}\label{4.2}
 \Vert \nabla \partial_{X}^2 u\Vert_{L^2}\leqslant C \Big( \Vert f\Vert_{L^2}+ \Vert \partial_X f \Vert_{L^2} +\Vert\partial^2_X f\Vert_{L^2}+\Vert f \Vert_{L^2}^2 + \Vert \partial_{X} f\Vert_{L^2}^2 + \Vert f\Vert_{L^2}^4 \Big)
,\end{equation}
where $C:=C(\Omega ,X)$.
\end{lemma}
\begin{proof}
Consider 
\begin{align}\label{higtanNS}
\partial_{X}^2\left[ (u\cdot\nabla)u -{\rm div\,}\left[\nu(x)\mathbf{D} u\right] +\nabla{p}\right] =\partial_{X}^2f\\
\partial_{X}^2 {\rm div\,}u=0
.\end{align}

In order to prove the main estimate we just take the second directional derivative from the nonlinear term of (\ref{NS}), the rest has been proved in Lemma \ref{l_hihgerRegSt}. 
\begin{equation}\label{nonlin}
\begin{split}
\partial_{X}^2\left[(u\cdot\nabla)u\right]&=\partial_{X}\biggl[ (\partial_{X} u\cdot \nabla)u+( u\cdot \nabla)\partial_{X}u-u^i \partial_{x_i}X^l \partial_{x_l} u^k \biggr]\\
&=(\partial_{X}^2 u\cdot \nabla)u+2(\partial_{X} u\cdot \nabla)\partial_{X}u-2 u \cdot \nabla X \cdot\nabla u+( u\cdot \nabla)\partial_{X}^2 u \\
&-\partial_{X} u\cdot \nabla X \cdot\nabla u- u \cdot X \nabla^2 X \cdot\nabla u - u \cdot\nabla X \cdot\nabla \partial_{X} u + u ( \nabla X )^2 \nabla u
.\end{split} 
\end{equation}
We test the above expression  by  $\partial_{X}^2 u$ and consider more precisely the most problematic one 
\begin{equation}\label{II}
I=\int_{\Omega} (\partial_{X}^2 u\cdot \nabla)u\,\partial_{X}^2 u \,dx \leqslant \Vert\partial_{X}^2 u\Vert^2_{L^4}\Vert \nabla u\Vert_{L^2}
.\end{equation}
According to (\ref{interpolation}), we have
\begin{equation}\label{interpolation2}
\Vert\partial_{X}^2 u\Vert_{L^4}\leqslant C \Vert \partial_{X}^2 u\Vert^{\frac{1}{2}}_{L^2}\Vert\partial_{X}^2 u\Vert^{\frac{1}{2}}_{BMO}
,\end{equation}
where 
\[\Vert\partial_{X}^2 u \Vert_{L^2(\Omega)} \leqslant C \Vert \nabla \partial_{X}u \Vert_{L^2(\Omega)}.\]
We implement (\ref{interpolation}) to (\ref{II}) and get that
\begin{equation}
\begin{split}
I&\leqslant C_1 \Vert u\Vert_{H^1}\Vert\nabla \partial_{X} u\Vert_{L^2}\Vert\partial_{X}^2 u\Vert_{BMO} \leqslant C_2 \Vert u\Vert_{H^1}\Vert\partial_{X} u\Vert_{H^1}\Vert\partial_{X}^2 u\Vert_{BMO}\\
&\leqslant 
C_2 \Vert u\Vert_{H^1}\Vert\partial_{X} u\Vert_{H^1}\Vert\partial_{X}^2 u\Vert_{H^1}
.\end{split}
\end{equation}

From the assumptions given vector field is smooth and we know that $u\in H^1(\Omega)$, $\partial_{X} u \in H^1(\Omega$). So without loss of generality, we could estimate the rest terms of (\ref{nonlin})
\begin{equation}
\begin{split}
\int_{\Omega}\partial_{X}^2\left[(u\cdot\nabla)u\right]\partial_{X}^2 u \, dx & \leqslant  C_3 \Vert u\Vert_{H^1}\Vert\partial_{X} u\Vert_{H^1}\Vert\partial_{X}^2 u\Vert_{H^1}+C_4 \Vert u\Vert_{H^1}^2 \Vert\partial_{X}^2 u\Vert_{H^1}\\
&+ C_5 \Vert\partial_{X} u\Vert_{H^1}^2\Vert\partial_{X}^2 u\Vert_{H^1}\\
& \leqslant (C_6 \Vert u\Vert_{H^1}^2 + C_7 \Vert\partial_{X} u\Vert_{H^1}^2)\Vert\partial_{X}^2 u\Vert_{H^1}
.\end{split} 
\end{equation}

Using Young's inequality with small $\epsilon$ and Lemma \ref{l_PropNS}, we get 
\begin{equation}
\begin{split}
\int_{\Omega}\partial_{X}^2\left[(u\cdot\nabla)u\right]\partial_{X}^2 u \, dx & \leqslant C_8 \Vert f \Vert_{L^2}^4 + C_{9} \Vert \partial_{X} f\Vert_{L^2}^4 +C_{10} \Vert f\Vert_{L^2}^8+\epsilon\Vert\partial_{X}^2 u\Vert_{H^1}^2 
.\end{split}
\end{equation}
From Lemma IX 1.2 [\ref{galdi_book}] and the same way as in (\ref{fcl_2HR}) we deduce that there exists a uniquely determined $\partial_{X}^2 p\in L^2 (\Omega)$ for Navier-Stokes system.

The estimate (\ref{3.62}) from the last step of the proof of Lemma \ref{higtan} combined with the above estimate for nonlinear term  give us an estimate of (\ref{higtanNS})
\begin{equation}\label{4.21}
\begin{split}
C \Vert \nabla \partial_{X} u\Vert_{L^2}^2 &\leqslant B_2\Vert f\Vert_{L^2}^2+F_2 \Vert \partial_X f \Vert_{L^2}^2 +A_{2}\Vert\partial^2_X f\Vert_{L^2}^2 + \epsilon \Vert \nabla \partial_{X} u \Vert_{L^2}^2\\
& +C_8 \Vert f \Vert_{L^2}^4 + C_{9} \Vert \partial_{X} f\Vert_{L^2}^4 +C_{10} \Vert f\Vert_{L^2}^8  
.\end{split}
\end{equation}
Taking $\epsilon$ in (\ref{4.21}) in such a way that $C > \epsilon $, we obtain (\ref{4.2}).

\end{proof}

\subsection{Proof of Theorem 2.3}
In order to prove Theorem \ref{thm_main} we need results stated in subsection 4.1. We assume that the approximate solution of problem (\ref{NS})-(\ref{NSdiv}) have tangential regularity (Lemma \ref{l_PropNS}) and higher tangential regularity (Lemma \ref{higtan}).   
\begin{proof} 
In general, we want to change the global system of coordinates $(x_1 ,x_2)$ to the local normal and tangent vector coordinates system    $(\tau , n)$ on $\partial S$. The tangent vector direction becomes as $y_1$ and the normal vector takes direction $y_2$. The most problematic part is to transfer the Navier-Stokes equations from the closed domain into the whole space.

 Thus consider, that Navier-Stokes equation is given in the neighbourhood $\Sigma$  of the boundary $\partial S$ s.t. there are open domains $S\subset S^{\prime} $ and $S^{\prime \prime}\subset S $ that is  $dist(x, x^{\prime})=\delta$ for $x\in \partial S$, $x^{\prime}\in \partial S^{\prime}$, and also $dist(x,x^{\prime \prime})=\delta $ for  $x\in \partial S$, $x^{\prime \prime}\in \partial S^{\prime \prime}$. Let us denote  the neighbourhood of $\partial S$ that is $S^{\prime}\setminus S^{\prime \prime}:= \Sigma$, and $\Omega^{(1)}:= \Omega \setminus (S\cup \Sigma )$, $\Omega^{(2)}:= S\setminus \Sigma$ (Fig.\ref{fig:Flow2}).  
 Without loss of generality, consider  $ \mathcal{O}: = W   \cap \Sigma  $ where $W $ is a plane that intersects with a part of  $\Sigma$. Let us fix $\epsilon$ and introduce notations
 \begin{equation}\label{omega1E}
 \Omega^{(1)}_{\epsilon}:=\lbrace x\in\Sigma\cup \Omega^{(1)} :\,\, dist(x,\Omega^{(1)})<\epsilon \rbrace
 \end{equation}
 \begin{equation}\label{omega2E}
 \Omega^{(2)}_{\epsilon}:=\lbrace x\in\Sigma\cup \Omega^{(2)} :\,\, dist(x,\Omega^{(2)})<\epsilon \rbrace
 .\end{equation}

\begin{figure}[h!]
 \centering
  \includegraphics[width=80mm]{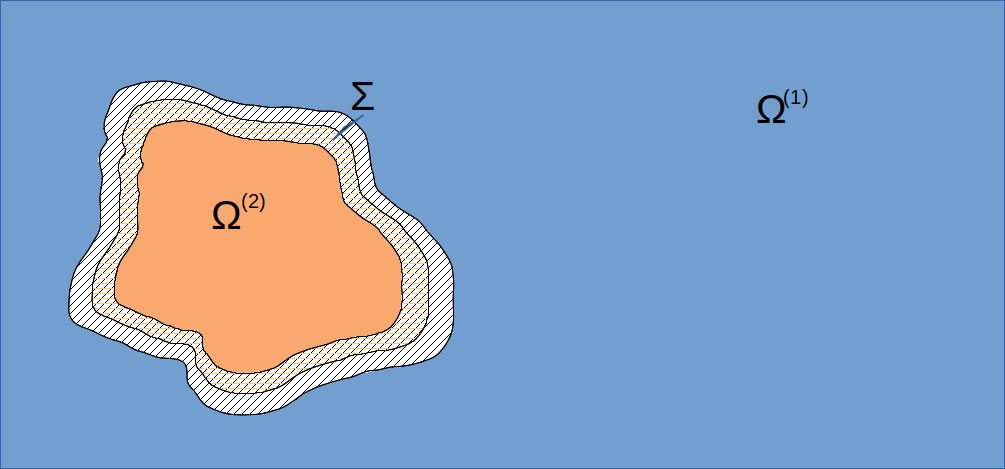}
  \caption{$\Sigma  $ neighborhood of the boundary of $S$.}
  \label{fig:Flow2}
\end{figure}

 There is a $\Phi$ - $C^2 $ diffeomorphism from $\mathcal{O}$ onto rectangle $V$, which is straightening out the boundary of obstacle $S$. In this regard, we apply a classical change of variables for the curvilinear system of coordinates (Fig.\ref{fig:Flow3}).

 \begin{figure}[h!]
 \centering
  \includegraphics[width=130mm]{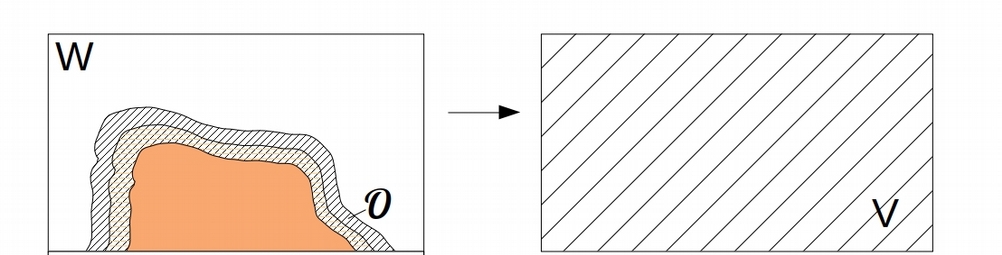}  \caption{Change of the system of coordinates. }
  \label{fig:Flow3}
\end{figure}


Change of variables takes form

\[y_i =\Phi_i (x),\,\, x_i =\Psi_i (y)\]
such that
\[ J_{\Psi}=\left( \frac{\partial \Psi_{i}}{\partial y_j}\right)_{i,j}\]
\[J_{\Phi}=\left( \frac{\partial \Phi_{i}}{\partial x_j}\right)_{i,j},\]
 with ${\rm det\,} J_{\Psi}(y)=1,\,\,\, \forall y \in \mathbb{R}^2 $. 
\[u(x)= J_{\Psi}(\Phi(x))U(\Phi(x))\]
\begin{equation}\label{4.17}
U(y)=J_{\Phi}(\Psi(y))u(\Psi(y))
,\end{equation}
i.e.,
\[U_i(y)=\sum_{j=1}^{2} \frac{\partial \Phi_i}{\partial x_j}u_j (\Psi (y))\]

\[P(y)=p(\Psi(y)).\]

The derivative along tangential field in the curvilinear system takes form
\[ \partial_{Y}:=\partial_{y_1} .\]

The jump of the viscosity field is transferred  along $y_2$ direction, so we get the dependence 
 \begin{equation}
\nu(y_2)= \begin{cases}
1, & y_2 >0\\
m, & y_2 <0
.\end{cases}
\end{equation}

We define the second order derivative operator as 
\[ \left[  \mathcal{L} U \right] _i=\sum_{j}  \frac{\partial}{\partial y_j }\left\lbrace\nu(y_2) \sum_{k}g^{jk}\frac{\partial U_i}{\partial y_k }+ 2 \nu(y_2)\sum_{l,k}g^{kj}\Gamma^i_{l k}U_l \right\rbrace, \]
\[ \left[ \mathcal{G} P\right]_i =\sum_j g^{ij} \frac{\partial P}{\partial y_j}  ,\]
and the convection term
\[\left[  \mathcal{N} U\right] _i=\sum_j U_j \frac{\partial U_i}{\partial y_j }+\sum_{j,k+1} \Gamma^i_{jk}U_j U_k.\]
Above, $\Gamma^i_{jk} $ are Christoffel symbols
\[\Gamma^k_{ij}=\frac{1}{2}\sum_{l}g^{kl}\left\lbrace \frac{\partial g_{il}}{\partial g_j}+\frac{\partial g_{jl}}{\partial y_i }-\frac{\partial g_{ij}}{\partial y_l}\right\rbrace ,\]
with contravariant vectors tensor
\[ g^{ij}=\sum_k \frac{\partial \Phi_i}{\partial x_k}\frac{\partial \Phi_j}{\partial x_k} ,\]
and covariant vectors tensor
\[ g_{ij}=\sum_k \frac{\partial \Psi_k}{\partial y_i}\frac{\partial \Psi_k}{\partial y_j}. \]
So, we have

\[{\rm div_y\,}U=0  \]
 by Corollary A.3. in [\ref{danchinmucha}] .

Extending all the fields by Sobolev extension s.t. $E U = U$ in $ V$ and $E U=0$ in $\mathbb{R}^2 \setminus V^{\prime}$, $V\subset V^{\prime}$, by  Theorem II.3.3 in [\ref{galdi_book}], we get Navier-Stokes equations in the curvilinear coordinates 

\begin{equation}
-\mathcal{L}U+\mathcal{N} U + \mathcal{G} P=F\;\;\; in\;\;\; \mathbb{R}^2
\end{equation}
\[{\rm div_{y} \,}U=0.\]

We rewrite the first row of the above equation 
\begin{equation}\label{N-S1row}
\begin{split}
\nu(y_2)&\partial_{y_1}\left\lbrace g^{11} \frac{ \partial U_1}{\partial{y_1}}+g^{12} \frac{ \partial U_1}{\partial{y_2}}+2\left[ g^{11}(\Gamma ^1_{11} U_1 +\Gamma ^1_{21} U_2)+g^{21}(\Gamma ^1_{12} U_1 +\Gamma ^1_{22} U_2)\right] \right\rbrace \\
& + \partial_{y_2}\left\lbrace \nu(y_2)\left(  g^{21} \frac{ \partial U_1}{\partial{y_1}}+g^{22} \frac{ \partial U_1}{\partial{y_2}}+2\left[ g^{12}(\Gamma ^1_{11} U_1 +\Gamma ^1_{21} U_2)+g^{22}(\Gamma ^1_{12} U_1 +\Gamma ^1_{22} U_2) \right] \right)\right\rbrace\\
& -U_1 \frac{\partial U_1}{\partial{y_1}}-U_2 \frac{\partial U_1}{\partial{y_2}}-\Gamma^1_{11} U_1^2-2 \Gamma^1_{12} U_{1} U_{2}-\Gamma^1_{22} U_{2}^2- g^{11}\frac{\partial P}{\partial {y_1}}-g^{12}\frac{\partial P}{\partial {y_2}}= -F_1
.\end{split}
\end{equation}
Since $\partial S \in C^2$ and compact, we have
\[  \vert g^{ij}\vert\leqslant K_0, \,\,\,\;\;\; \,\vert\frac{\partial g^{ij}}{\partial y_j}\vert\leqslant K_1 \]
\[ \Big\vert\frac{\partial \Gamma_{jk}^{i}}{\partial y_s}\Big\vert\leqslant K_2, \,\,\,\;\;\;\,\vert\Gamma_{jk}^{i}\vert\leqslant K_3 .\]
Recall, the space dimension  $n=1$ and  $p=2$, we use an estimate for the convection term as follows
\begin{equation}
\Vert U\nabla U \Vert_{L^2 (\mathbb{R})} \leqslant  C \Vert U \Vert_{H^1 (\mathbb{R})}^2
.\end{equation}
We know $u\in H^1(\Omega)$, which implies $U\in H^1(\mathbb{R}^2)$.
By Lemma \ref{l_PropNS} and (\ref{4.17}) we deduce that $\nabla \partial_{y_1} U \in L^2(\mathbb{R}_{y_2};L^2(\mathbb{R}_{y_2})) $.
 We take $L^2(\mathbb{R}_{y_2};L^2(\mathbb{R}_{y_1})) $ norm of (\ref{N-S1row}). Because the differentiation in $y_1$ direction is well defined we transfer all terms of the LHS to the RHS, except the $2^{nd}$ one. Also, using Cauchy-Schwartz, Poincar\'e inequalities, we get              
\begin{equation}\label{4.25}
\begin{split}
\Big[\int_{L^2(\mathbb{R}_{y_2})} &\int_{L^2(\mathbb{R}_{y_1})} \Big\vert \partial_{y_2}\left\lbrace \nu(y_2)\left(  g^{21} \frac{ \partial U_1}{\partial{y_1}}+g^{22} \frac{ \partial U_1}{\partial{y_2}}+ R\right) \right\rbrace \Big\vert^2 d\,y_1 d\,y_2 \Big]^{\frac{1}{2}} \\
&\leqslant \overline{K}_0 \Big( \Vert \nu(y_2)  U_{1,1}\Vert_{L^2(\mathbb{R}_{y_2};H^1(\mathbb{R}_{y_1}))}+\Vert \nu(y_2)  U_{1,2}\Vert_{L^2(\mathbb{R}_{y_2};H^1(\mathbb{R}_{y_1}))}\Big)\\
& + \overline{K}_1 \Big( \Vert \nu(y_2)  U_{1}\Vert_{L^2(\mathbb{R}_{y_2};H^1(\mathbb{R}_{y_1}))}+\Vert \nu(y_2)  U_{2}\Vert_{L^2(\mathbb{R}_{y_2};H^1(\mathbb{R}_{y_1}))}\Big) 
 +\Vert F_1 \Vert_{L^2(\mathbb{R}_{y_2};H^1(\mathbb{R}_{y_1}))} \\
 &+(1+K_3 ) \Big( \Vert U_1 \Vert_{L^2(\mathbb{R}_{y_2};H^1(\mathbb{R}_{y_1}))}^2+\Vert U_2 \Vert_{L^2(\mathbb{R}_{y_2};H^1(\mathbb{R}_{y_1}))} \Vert U_1\Vert_{L^2(\mathbb{R}_{y_2};H^1(\mathbb{R}_{y_1}))}\Big)\\
 &+ K_0 \Vert \nabla P\Vert_{L^2(\mathbb{R}_{y_2};L^2 (\mathbb{R}_{y_1}))}+K_3 \Vert U_2\Vert^2_{L^2(\mathbb{R}_{y_2};L^2 (\mathbb{R}_{y_1}))}
,\end{split}
\end{equation}
where $ R:= 2\left[ g^{12}(\Gamma ^1_{11} U_1 +\Gamma ^1_{21} U_2)+g^{22}(\Gamma ^1_{12} U_1 +\Gamma ^1_{22} U_2)\right] $.

It's obvious that the LHS of (\ref{4.25})  $\partial_{y_2}\left\lbrace \nu(y_2)\left(  g^{21} \frac{ \partial U_1}{\partial{y_1}}+g^{22} \frac{ \partial U_1}{\partial{y_2}}+R \right) \right\rbrace   \in L^{2} (\mathbb{R}_{x_2};L^2 (\mathbb{R}_{x_1}))$.

In situation when we have $\partial_{y_1}^2 F\in L^2(\mathbb{R}_{y_2}; L^2 (\mathbb{R}_{y_1}))$ and $\nabla \partial_{y_1} P \in L^2(\mathbb{R}_{y_2}; L^2(\mathbb{R}_{y_1}))$ from Lemma \ref{higtan} it is obvious that $\partial_{y_1}^2 (U\nabla U)\in L^2(\mathbb{R}_{y_2};L^2(\mathbb{R}_{y_1}))$), we deduce that $\nabla \partial_{y_1}^2 U \in L^2(\mathbb{R}^2)$. 
If we differentiate by $y_1$ both sides of (\ref{N-S1row}), then obviously  we get that 

\begin{equation}\label{4.26}
\begin{split}
\Big\Vert \partial_{y_2}\Big\lbrace &\nu(y_2)\left(  g^{21} \frac{ \partial U_1}{\partial{y_1}}+g^{22} \frac{ \partial U_1}{\partial{y_2}}+R \right)  \Big\rbrace \Big\Vert_{L^2(\mathbb{R}_{y_2};H^1(\mathbb{R}_{y_1}))} \\
&\leqslant \overline{K}_0 \left( \Vert \nu(y_2)  U_{1,11}\Vert_{L^2(\mathbb{R}_{y_2};H^1(\mathbb{R}_{y_1}))}+\Vert \nu(y_2)  U_{1,21}\Vert_{L^2(\mathbb{R}_{y_2};H^1(\mathbb{R}_{y_1}))}\right)  \\
&+ \overline{K}_1 \left( \Vert \nu(y_2)  U_{1,1}\Vert_{L^2(\mathbb{R}_{y_2};H^1(\mathbb{R}_{y_1}))}+\Vert \nu(y_2)  U_{2,1}\Vert_{L^2(\mathbb{R}_{y_2};H^1(\mathbb{R}_{y_1}))}\right)
 +\Vert F_{1,1} \Vert_{L^2(\mathbb{R}_{y_2};H^1(\mathbb{R}_{y_1}))} \\
 &+(1+K_3 ) \left( \Vert U_{1,1} \Vert_{L^2(\mathbb{R}_{y_2};H^1(\mathbb{R}_{y_1}))}^2
 +\Vert U_{2,1}\Vert_{L^2(\mathbb{R}_{y_2};H^1(\mathbb{R}_{y_1}))} \Vert U_{1,1}\Vert_{L^2(\mathbb{R}_{y_2};H^1(\mathbb{R}_{y_1}))}\right)\\
 &+ K_0 \Vert \nabla P_{,1}\Vert_{L^2(\mathbb{R}_{y_2};L^2 (\mathbb{R}_{y_1}))}+K_3 \Vert U_{2,1} \Vert^2_{L^2(\mathbb{R}_{y_2};L^2 (\mathbb{R}_{y_1}))}
,\end{split}
\end{equation}
i.e. we have $\partial_{y_2}\left\lbrace \nu(y_2)\left(  g^{21} \frac{ \partial U_1}{\partial{y_1}}+g^{22} \frac{ \partial U_1}{\partial{y_2}}+R \right) \right\rbrace   \in L^{2} (\mathbb{R}_{x_2};H^1 (\mathbb{R}_{x_1}))$, that means
$$\nu(y_2)\left(  g^{21} \frac{ \partial U_1}{\partial{y_1}}+g^{22} \frac{ \partial U_1}{\partial{y_2}}+ R \right) \in H^1(\mathbb{R}_{y_2};H^1(\mathbb{R}_{y_1}))\hookrightarrow L^{\infty}(\mathbb{R}_{y_2};H^1 (\mathbb{R}_{y_1})).$$

We split the above norm into two parts using the triangle inequality  
\begin{equation}
\begin{split}
\Vert \nu(y_2) g^{22} U_{1,2}\Vert_{L^{\infty}(\mathbb{R}_{y_2};H^1 (\mathbb{R}_{y_1}))}&\leqslant
\Big\Vert\nu(y_2)\left(  g^{21}  U_{1,1} +g^{22} U_{1,2}+ R \right) \Big\Vert_{L^{\infty}(\mathbb{R}_{y_2};H^1 (\mathbb{R}_{y_1}))}\\
&+\Big\Vert\nu(y_2)\left(  g^{21}  U_{1,1} + R\right) \Big \Vert_{L^{\infty}(\mathbb{R}_{y_2};H^1 (\mathbb{R}_{y_1}))}
,\end{split} 
\end{equation}
and use it in (\ref{4.26}).
By embedding  $L^{\infty}(\mathbb{R}_{y_2};H^1(\mathbb{R}_{y_1}))\hookrightarrow L^{\infty}(\mathbb{R}_{y_2};L^{\infty} (\mathbb{R}_{y_1}))$, we could bound the above inequality from below and get 

\begin{equation}
\begin{split}\Vert \nu(y_2) & g^{22} U_{1,2}\Vert_{L^{\infty}(\mathbb{R}_{y_2};L^{\infty} (\mathbb{R}_{y_1}))}\leqslant 
\overline{K}_0 \left( \Vert \nu(y_2)  U_{1,11}\Vert_{L^2(\mathbb{R}_{y_2};H^1(\mathbb{R}_{y_1}))}+\Vert \nu(y_2)  U_{1,21}\Vert_{L^2(\mathbb{R}_{y_2};H^1(\mathbb{R}_{y_1}))}\right) \\
&+ \overline{K}_1 \left( \Vert \nu(y_2)  U_{1,1}\Vert_{L^2(\mathbb{R}_{y_2};H^1(\mathbb{R}_{y_1}))}+\Vert \nu(y_2)  U_{2,1}\Vert_{L^2(\mathbb{R}_{y_2};H^1(\mathbb{R}_{y_1}))}\right)
 +\Vert F_{1,1} \Vert_{L^2(\mathbb{R}_{y_2};H^1(\mathbb{R}_{y_1}))}\\
 & +(1+K_3 ) \left( \Vert U_{1,1} \Vert_{L^2(\mathbb{R}_{y_2};H^1(\mathbb{R}_{y_1}))}^2 +\Vert U_{2,1}\Vert_{L^2(\mathbb{R}_{y_2};H^1(\mathbb{R}_{y_1}))} \Vert U_{1,1}\Vert_{L^2(\mathbb{R}_{y_2};H^1(\mathbb{R}_{y_1}))}\right)\\
 &+ K_0 \Vert \nabla P_{,1}\Vert_{L^2(\mathbb{R}_{y_2};L^2(\mathbb{R}_{y_1}))}+K_3 \Vert U_{2}\Vert_{L^2(\mathbb{R}_{y_2};H^1 (\mathbb{R}_{y_1}))}\\
 &+\Vert\nu(y_2)\left(  g^{21}  U_{1,1}+ R \right) \Vert_{L^{\infty}(\mathbb{R}_{y_2};H^1 (\mathbb{R}))},
\end{split} 
\end{equation}

and it's also true that we have
\begin{equation}
\Vert U_{1,2}\Vert_{L^{\infty}(\mathbb{R}_{y_2};L^{\infty} (\mathbb{R}_{y_1}))}\leqslant \Big\Vert \dfrac{1}{\nu(y_2)  g^{22}}\Big\Vert_{L^{\infty}(\mathbb{R}_{y_2};L^{\infty} (\mathbb{R}_{y_1}))}\Vert \nu(y_2)  g^{22} U_{1,2}\Vert_{L^{\infty}(\mathbb{R}_{y_2};L^{\infty} (\mathbb{R}_{y_1}))}
,\end{equation}
 thus we have a priori estimate for $U_{1,2}$

\begin{equation}
\begin{split}
\Vert U_{1,2}&\Vert_{L^{\infty}(\mathbb{R}^2)}\leqslant \Big\Vert \dfrac{1}{\nu(y_2)  g^{22}}\Big\Vert_{L^{\infty}(\mathbb{R}^2)}
 \lbrace \overline{K}_0 \left( \Vert \nu(y_2)  U_{1,1}\Vert_{L^2(\mathbb{R}_{y_2};H^1(\mathbb{R}_{y_1}))}+\Vert \nu(y_2)  U_{1,2}\Vert_{L^2(\mathbb{R}_{y_2};H^1(\mathbb{R}_{y_1}))}\right) \\
&+ \overline{K}_1 \left( \Vert \nu(y_2)  U_{1}\Vert_{L^2(\mathbb{R}_{y_2};H^1(\mathbb{R}_{y_1}))}+\Vert \nu(y_2)  U_{2}\Vert_{L^2(\mathbb{R}_{y_2};H^1(\mathbb{R}_{y_1}))}\right)
 +\Vert F_1 \Vert_{L^2(\mathbb{R}_{y_2};H^1(\mathbb{R}_{y_1}))}\\
 & +(1+K_3 ) \left( \Vert U_1 \Vert_{L^2(\mathbb{R}_{y_2};H^1(\mathbb{R}_{y_1}))}^2 +\Vert U_2\Vert_{L^2(\mathbb{R}_{y_2};H^1(\mathbb{R}_{y_1}))} \Vert U_1\Vert_{L^2(\mathbb{R}_{y_2};H^1(\mathbb{R}_{y_1}))}\right)\\
 &+\Vert\nu(y_2)\left(  g^{21}  U_{1,1}+R \right) \Vert_{L^{\infty}(\mathbb{R}_{y_2};H^1 (\mathbb{R}_{y_1}))}+K_3 \Vert U_{2}\Vert_{L^2(\mathbb{R}_{y_2};H^1 (\mathbb{R}_{y_1}))}+ K_0 \Vert  P\Vert_{L^2(\mathbb{R}_{y_2};H^1(\mathbb{R}_{y_1}))} \rbrace
.\end{split} 
\end{equation}
So, we have that $U_{1,2}\in L^{\infty}(\mathbb{R}_{y_2};L^{\infty} (\mathbb{R}_{y_1})) $, the same is true for $U_{1,1}\in L^{\infty}(\mathbb{R}_{y_2};L^{\infty} (\mathbb{R}_{y_1}))$.

Now, we consider the second row of the Navier-Stokes  equation

\begin{equation}
\begin{split}
\nu(y_2)\partial_{y_1}&\left\lbrace g^{11} \frac{ \partial U_2}{\partial{y_1}}+g^{12} \frac{ \partial U_2}{\partial{y_2}}+2\left[ g^{11}(\Gamma ^2_{11} U_1 +\Gamma ^2_{21} U_2)+g^{21}(\Gamma ^2_{12} U_1 +\Gamma ^2_{22} U_2)\right] \right\rbrace \\
 + &\partial_{y_2}\left\lbrace \nu(y_2)\left(  g^{21} \frac{ \partial U_2}{\partial{y_1}}+g^{22} \frac{ \partial U_2}{\partial{y_2}}+2\left[ g^{12}(\Gamma ^2_{11} U_1 +\Gamma ^2_{21} U_2)+g^{22}(\Gamma ^2_{12} U_1 +\Gamma ^2_{22} U_2)\right) \right] \right\rbrace\\
& -U_1 \frac{\partial U_2}{\partial{y_1}}-U_2 \frac{\partial U_2}{\partial{y_2}}-\Gamma^2_{11} U_1^2-\Gamma^2_{22} U_2^2-2 \Gamma^2_{12} U_{1} U_{2}- g^{21}\frac{\partial P}{\partial {y_1}}-g^{22}\frac{\partial P}{\partial {y_2}}= -F_2
.\end{split}
\end{equation}

Let us denote $ R_1 := 2\left[ g^{12}(\Gamma ^2_{11} U_1 +\Gamma ^2_{21} U_2)+g^{22}(\Gamma ^2_{12} U_1 +\Gamma ^2_{22} U_2)\right] $.

From the above equation, we take the  $L^2(\mathbb{R}_{y_2};L^2(\mathbb{R}_{y_1}))$ norm and bound as in the previous case

\begin{equation}
\begin{split}
 \Big\Vert \partial_{y_2}\Big\lbrace &\nu(y_2)\left(  g^{21} \frac{ \partial U_2}{\partial{y_1}}+g^{22} \frac{ \partial U_2}{\partial{y_2}}+ R_1 \right) \Big\rbrace \Big\Vert_{L^2(\mathbb{R}_{y_2};L^2(\mathbb{R}_{y_1}))}   \\
&\leqslant  \overline{K}_0 \left( \Vert \nu(y_2)  U_{2,1}\Vert_{L^2(\mathbb{R}_{y_2};H^1(\mathbb{R}_{y_1}))}+\Vert \nu(y_2)  U_{2,2}\Vert_{L^2(\mathbb{R}_{y_2};H^1(\mathbb{R}_{y_1}))}\right)  \\
&+\overline{K}_1 \left( \Vert \nu(y_2)  U_{1}\Vert_{L^2(\mathbb{R}_{y_2};H^1(\mathbb{R}_{y_1}))}+\Vert \nu(y_2)  U_{2}\Vert_{L^2(\mathbb{R}_{y_2};H^1(\mathbb{R}_{y_1}))}\right)\\
 & +(1+K_3 ) \left( \Vert U_2 \Vert_{L^2(\mathbb{R}_{y_2};H^1(\mathbb{R}_{y_1}))}^2+\Vert U_2\Vert_{L^2(\mathbb{R}_{y_2};H^1(\mathbb{R}_{y_1}))} \Vert U_1\Vert_{L^2(\mathbb{R}_{y_2};H^1(\mathbb{R}_{y_1}))}\right)\\
& + K_0 \Vert \nabla P\Vert_{L^2(\mathbb{R}_{y_2};H^1(\mathbb{R}_{y_1}))}+\Vert F_2 \Vert_{L^2(\mathbb{R}_{y_2};H^1(\mathbb{R}_{y_1}))}+K_3 \Vert U_1 \Vert_{L^2(\mathbb{R}_{y_2};L^2(\mathbb{R}_{y_1}))}^2
.\end{split}
\end{equation}

Differentiating the above estimate by $\partial_{y_1}$ it is obvious that, $$\partial_{y_2}\left\lbrace \nu(y_2)\left(  g^{21} \frac{ \partial U_2}{\partial{y_1}}+g^{22} \frac{ \partial U_2}{\partial{y_2}}+R_1 \right) \right\rbrace \in L^2(\mathbb{R}_{y_2}; H^1(\mathbb{R}_{y_1}))\hookrightarrow L^2(\mathbb{R}_{y_2};L^{\infty}(\mathbb{R}_{y_1})).$$ Consequently, $\nu(y_2)\left(  g^{21} \frac{ \partial U_2}{\partial{y_1}}+g^{22} \frac{ \partial U_2}{\partial{y_2}}+R_1 \right) \in H^1 (\mathbb{R}_{y_2}; H^1(\mathbb{R}_{y_1}))\hookrightarrow L^{\infty}(\mathbb{R}_{y_2};L^{\infty}(\mathbb{R}_{y_1}))$.\\ 
In order to bound $U_{2,2}$, we repeat the same procedure as in the first case
\begin{equation}
\begin{split}
\Vert  U_{2,2}&\Vert_{L^{\infty}(\mathbb{R}^2)}\leqslant \Big\Vert \dfrac{1}{\nu(y_2)  g^{22}}\Big\Vert_{L^{\infty}(\mathbb{R}^2)}
\lbrace\overline{K}_0 \left( \Vert \nu(y_2)  U_{2,11}\Vert_{L^2(\mathbb{R}_{y_2};H^1(\mathbb{R}_{y_1}))}+\Vert \nu(y_2)  U_{2,21}\Vert_{L^2(\mathbb{R}_{y_2};H^1(\mathbb{R}_{y_1}))}\right) \\
+& \overline{K}_1 \left( \Vert \nu(y_2)  U_{1,1}\Vert_{L^2(\mathbb{R}_{y_2};H^1(\mathbb{R}_{y_1}))}+\Vert \nu(y_2)  U_{2,1}\Vert_{L^2(\mathbb{R}_{y_2};H^1(\mathbb{R}_{y_1}))} \right)
 +\Vert F_{2,1} \Vert_{L^2(\mathbb{R}_{y_2};H^1(\mathbb{R}_{y_1}))}\\
  +&(1+K_3 ) \left( \Vert U_{2,1} \Vert_{L^2(\mathbb{R}_{y_2};H^1(\mathbb{R}_{y_1}))}^2 +\Vert U_{2,1}\Vert_{L^2(\mathbb{R}_{y_2};H^1(\mathbb{R}_{y_1}))} \Vert U_{1,1}\Vert_{L^2(\mathbb{R}_{y_2};H^1(\mathbb{R}_{y_1}))}\right)\\
 &+ K_0 \Vert \nabla P_{,1}\Vert_{L^2(\mathbb{R}_{y_2};H^1(\mathbb{R}_{y_1}))}+K_3 \Vert U_1 \Vert_{L^2(\mathbb{R}_{y_2};H^1 (\mathbb{R}_{y_1}))}^2\\
&+\Vert\nu(y_2)\left(  g^{21}  U_{2,1}+R_1\right)  \Vert_{L^{\infty}(\mathbb{R}_{y_2};H^1 (\mathbb{R}_{y_1}))}\rbrace
.\end{split} 
\end{equation}

 It follows that $U_{2,2}\in L^{\infty}(\mathbb{R}_{y_2};L^{\infty} (\mathbb{R}_{y_1})) $, the same is true for $U_{2,1}\in L^{\infty}(\mathbb{R}_{y_2};L^{\infty} (\mathbb{R}_{y_1}))$. So, we have that $\nabla U \in L^{\infty}(\mathbb{R}_{y_2};L^{\infty} (\mathbb{R}_{y_1}))$.
 By extension theorem and using (\ref{4.17}) that is  $C_1\Vert\nabla U \Vert \leqslant\Vert \nabla u\Vert\leqslant C_2 \Vert \nabla U\Vert $, we conclude that $\nabla u \in L^{\infty}(\mathcal{O}) $. It's obvious that for the other part of neighborhood we get the same, so we deduce that $\nabla u\in L^{\infty}(\Sigma)$, the conclusion follows.

\end{proof}

The next lemma gives a higher regularity of $u$ in the $\Omega^{(1)}\cup \Omega^{(2)}$

\begin{lemma}
Let $\Omega \subset \mathbb{R}^2 $ be an open bounded domain of class $C^2$.  And let $u$ satisfy Navier-Stokes system of equations (\ref{NS})-(\ref{NSdiv}), $q > n$. If $f \in L^{q}(\Omega)$ then

 $\nabla u\in L^{\infty}(\Omega^{(1)}\cup\Omega^{(2)})$ and there exists $p\in W^{1,q}(\Omega^{(1)}\cup\Omega^{(2)})$ such that (\ref{NS}) is satisfies a.e.
\end{lemma}

\begin{proof}
 We have
\begin{equation}\label{visc2}
\nu(x)= \begin{cases}
1, &\mathbf{x}\in \Omega^{(1)} \\
m, &\mathbf{x}\in \Omega^{(2)}
.\end{cases}
\end{equation}

We localize the problem, take "cut off" function $\eta\in C_0^{\infty}(\mathbb{R}^2)$ such that 
\begin{equation}
\eta^{(1)} = \begin{cases}
1, &\mathbf{x}\in{\Omega^{(1)}}\\
0, &\mathbf{x}\in{\mathbb{R}^2 \setminus{\Omega^{(1)}_{\epsilon}}}
.\end{cases}
\end{equation}

Putting $w =u\eta^{(1)}$, $\pi=p\eta^{(1)}$, and taking into account (\ref{visc}) we have that $w$ and $\pi$ satisfies the following problem

\begin{equation}\label{eq:stokes_eq}
-{\rm div\,}\left[\mathbf{D }w\right] +u \cdot\nabla w+\nabla{\pi}=F\,\;\;\; in \,\, \Omega^{(1)}_{\epsilon}
\end{equation}

\begin{equation}\label{eq:div_eq}
{\rm div\,}w = g
\end{equation}
and
\[w=0 \;\;\, in \;\;\, \partial \Omega^{(1)}_{\epsilon},\]
where
 
\[
F=\eta^{(1)} f+p \nabla \eta^{(1)} -u \,{\rm div\,}(\nabla\eta^{(1)})- (\mathbf{D}u)(\nabla \eta^{(1)})^{\intercal}-(\nabla \eta^{(1)} \cdot \nabla)u-(u \cdot \nabla)(\nabla\eta^{(1)})+u\cdot \nabla \eta^{(1)} u  
\]
i.e.

\begin{equation}\label{F}
F=\eta^{(1)} f+p \nabla \eta^{(1)} -u\,\Delta \eta^{(1)} -u \nabla^2 \eta^{(1)}- (\nabla^{T}u)(\nabla \eta^{(1)}) -2\nabla u \nabla \eta^{(1)}+ u\cdot \nabla \eta^{(1)} u    
\end{equation}
\[g=\nabla \eta^{(1)} \cdot u .\]

Recalling the property of $\eta$ that is
$$\vert \nabla \eta^{(1)}\vert \leqslant C .$$
In some sense, we will repeat the proof according to Galdi [\ref{galdi_book}]. We want to show that
\begin{equation}\label{regularity_up}
(u,p)\in W^{2,q}(\Omega^{(1)}_{\epsilon})\times W^{1,q}(\Omega^{(1)}_{\epsilon}) \;\;\; for\;\;\; 1<q<\infty
.\end{equation}

We know that there is    
\[u\in W^{1,2}(\Omega^{(1)}_{\epsilon})\;\;\;\;\;  (*) \]
that satisfies Navier-Stokes system (\ref{NS}).
By the Lemma IX.2.1 ([\ref{galdi_book}]) and Lemma 2.1., we deduce there is $p\in L^2(\Omega^{(1)}_{\epsilon})$.
Furthermore, the embedding theorem, that furnishes recurrence relation for the exponents s.t. $u\in L^{2 t_k\diagup (2-t_k)}(\Omega^{(1)}_{\epsilon})$,
and (\ref{F}), allows us to conclude $F\in L^{q}(\Omega^{(1)}_{\epsilon})$, $g\in W^{1,q}(\Omega^{(1)}_{\epsilon})$ for $q\in (1,2)$.

\begin{equation}\label{4.44}
\Vert F\Vert_{q}\leqslant  c_1 \Vert f\Vert_q +c_2 \Vert p\Vert_q +c_3\Vert \nabla u\Vert_{q} +c_3\Vert u\Vert_{q}^2
.\end{equation}
So that by Lemma IX.5.1 and (*) we want to prove (\ref{regularity_up}).
 Assume next $q\geq n$. Then, the (*) is satisfied for all $q\in (1,n)$, and so, by the embedding theorem, we have
\[u\in L^{\infty}(\Omega^{(1)}_{\epsilon})\cap W^{1,t}(\Omega^{(1)}_{\epsilon})\, ,\;\;\, p\in L^{t}(\Omega^{(1)}_{\epsilon})\;\;\; for\;\;\; all \;\; \; t\in (1,\infty),\] 
yielding,   
\[u\cdot \nabla u \in L^{q}(\Omega^{(1)}_{\epsilon}).\]
 From the interior estimates for the Stokes problem proved in Theorem IV.4.1 [\ref{galdi_book}] follows that $u\in W^{2,q}(\Omega^{(1)}_{\epsilon})$ and $p\in W^{1,q}(\Omega^{(1)}_{\epsilon})$.

Denote $F^{\prime}= F-u\cdot \nabla w$,
for $q\geqslant n$, $1< q <\infty$.
 By the H\"older inequality and recalling the properties of $\eta$ it follows that
\[\Vert u \cdot \nabla w\Vert_{q}\leqslant C \Vert u\Vert_{r}\Vert \nabla u\Vert_{s} \leqslant C \Vert u\Vert_{1,q}\;\;\, in \;\;\, \Omega^{(1)}_{\epsilon}.\]
So, we have $F^{\prime} \in L^{q}(\Omega^{(1)}_{\epsilon})$, using  Lemma IX.5.1 [\ref{galdi_book}], we deduce that  
$w\in W^{2,q}(\Omega^{(1)}_{\epsilon})$ , $\pi\in W^{1,q}(\Omega^{(1)}_{\epsilon})$ and satisfies an estimate
\begin{equation}
\Vert w\Vert_{2,q,\Omega^{(1)}_{\epsilon}}+\Vert \pi\Vert_{1,q,\Omega^{(1)}_{\epsilon}}\leqslant C (\Vert F\Vert_{q,\Omega^{(1)}}+ \Vert g \Vert_{1,q,\Omega^{(1)}_{\epsilon}})
\end{equation}
also, using the Theorem IX.5.1 and properties of $\eta$ we derive
\begin{equation}
\Vert u\Vert_{2,q,\Omega^{(1)}_{\epsilon}}+\Vert p\Vert_{1,q,\Omega^{(1)}_{\epsilon}}\leqslant C_1 (\Vert F\Vert_{q,\Omega^{(1)}_{\epsilon}}+ \Vert g \Vert_{1,q,\Omega^{(1)}_{\epsilon}})
,\end{equation}
also,                                                                                                                                                                                                                                                                                                                                                                                                                                                                                                                                                        we could write
\begin{equation}
\Vert \nabla u\Vert_{1,q,\Omega^{(1)}_{\epsilon}}+\Vert p\Vert_{1,q,\Omega^{(1)}_{\epsilon}} \leqslant C_1 (\Vert F\Vert_{q,\Omega^{((1)}_{\epsilon}}+ \Vert g \Vert_{1,q,\Omega^{(1)}_{\epsilon}})
.\end{equation}

For the case $q>n$, we have embedding $W^{1,q}(\Omega^{(1)}_{\epsilon})\hookrightarrow L^{\infty}(\Omega^{(1)}_{\epsilon})$, thus $\nabla u\in L^{\infty}(\Omega^{(1)}_{\epsilon})$,

 \begin{equation}\label{4.48}
\Vert \nabla u\Vert_{L^{\infty}(\Omega^{(1)}_{\epsilon})}+\Vert p\Vert_{1,q,\Omega^{(1)}_{\epsilon}}\leqslant C_2 (\Vert F\Vert_{q,\Omega^{(1)}_{\epsilon}}+ \Vert g \Vert_{1,q,\Omega^{(1)}_{\epsilon}})
.\end{equation}
 Let us now consider the other "cut off" function $\eta^{(2)}$ such that,

\begin{equation}
\eta^{(2)} = \begin{cases}
1, &\mathbf{x}\in{\Omega^{(2)}}\\
0, &\mathbf{x}\in{\mathbb{R}^2 \setminus{\Omega^{(2)}_{\epsilon}}}
.\end{cases}
\end{equation}

We localize the problem (4.2)-(4.3), and put $v =u\eta^{(2)}$, $P=p\eta^{(2)}$, using the properties of the function $\eta^{(2)}$ we get

\begin{equation}\label{eq:stokes_eq}
-m\,{\rm div\,}\left[\mathbf{D }v\right] +u \cdot\nabla v+\nabla{\pi}=F\,\;\;\; in \,\, \Omega^{(2)}_{\epsilon}
\end{equation}

\begin{equation}\label{eq:div_eq}
{\rm div\,}v = g
,\end{equation}
and
\[w=0 \;\;\, in \;\;\, \partial \Omega^{(2)}_{\epsilon}.\]

 From the bounds of the first case, using that $F\in L^q (\Omega^{(2)}_{\epsilon})$, $g\in W^{1,q}(\Omega^{(2)}_{\epsilon}) $, and using Lemma IX.5.1, Theorem IX.5.1 ([\ref{galdi_book}]) we get an estimate

\begin{equation}
\Vert \nabla u\Vert_{1,q,\Omega^{(2)}_{\epsilon}}+\frac{1}{m}\Vert p\Vert_{1,q,\Omega^{(2)}_{\epsilon}}\leqslant \frac{C_1}{m} (\Vert F\Vert_{q,\Omega^{(2)}_{\epsilon}}+ \Vert g \Vert_{1,q,\Omega^{(2)}_{\epsilon}})
\end{equation}
using the embedding to $W^{1,q}(\Omega^{(2)}_{\epsilon})\hookrightarrow L^{\infty}(\Omega^{(2)}_{\epsilon})$,
\begin{equation}
\Vert \nabla u\Vert_{L^{\infty}(\Omega^{(2)}_{\epsilon})}+\frac{1}{m}\Vert p\Vert_{1,q,\Omega^{(2)}_{\epsilon}}\leqslant \frac{C_3}{m} (\Vert F\Vert_{q,\Omega^{(2)}_{\epsilon}}+ \Vert g \Vert_{1,q,\Omega^{(2)}_{\epsilon}})
.\end{equation}
 So, from (\ref{4.44}) and (\ref{4.48}) we conclude that $\nabla u\in L^{\infty}(\Omega^{(1)}_{\epsilon}\cup\Omega^{(2)}_{\epsilon})$.

\end{proof}

 \section{Numerical simulations}
 In this section, we will illustrate Theorem \ref{limit} with some numerical simulations, and show that our approximate problem (\ref{NS}-\ref{visc}) has  a potential application in practice.
    
  We conduct two numerical experiments, where we consider a smooth obstacle (a half ball) and an obstacle with edges (a wall). We do a number of tests with increasing value of viscosity $m$ (see \ref{visc}). It turns out that for moderately high values of $m$ the solutions are almost identical to the real rigid obstacle problem.     
  We consider a rectangular channel flow problem in $\Omega\setminus S$ with fixed rigid obstacle $S$ touching the boundary of $\Omega$   ([\ref{turek}], [\ref{lang}]). We refer to it as  a "real obstacle" problem and denote the velocity by $u_r$ ( see Table \ref{table:2}). The experimental data and the geometry of channel flow are inherited from the Turek's benchmark [\ref{turek}], test case 2D-2, with two exceptions the channel's length equals  $L=1.2$, and the obstacles are of different shape and touch the boundary (Fig.\ref{fig:Flow1}-\ref{fig:Flow2}]). The experiments with half ball obstacle is illustrated in Fig.\ref{Comp_1} with $r=0.15$ and center at $(0,4;0,0)$.
And the experiments with the wall obstacle is in Fig.\ref{Comp_2} with height $h=0,16$, width $w=0,1$ and center of symmetry $x=0,4$. 
   We assume a parabolic velocity profile at the inlet and Dirichlet boundary condition on the boundary.
  We compare the solution of the real obstacle problem with the solution of the approximate problem (\ref{NS})-(\ref{visc}), where the fixed obstacle domain $S$ is filled with highly viscous fluid of viscosity $m$ (see Table \ref{table:2}).  
  
   The results have been computed with the FEniCS package [\ref{logg}] using the incremental pressure correction scheme to solve the problem ([\ref{goda}]).


\begin{figure}[!tbp]
\centering
\begin{minipage}[b]{0.45\textwidth}
\includegraphics[width=\textwidth]{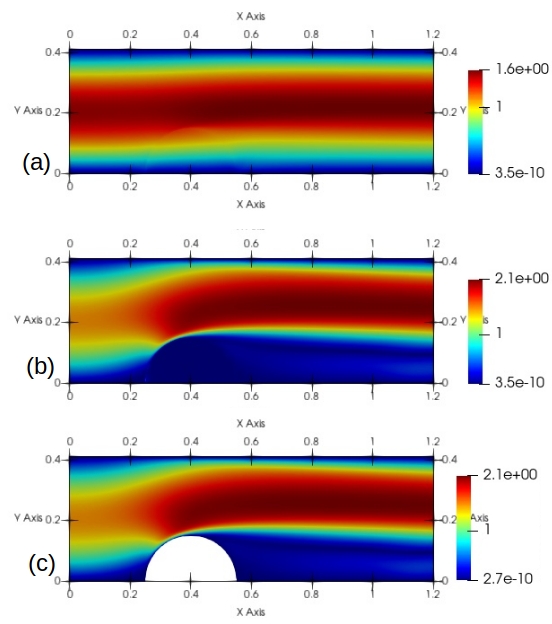}
\caption{ The magnitude of the velocity for the half ball test problem. The (a),(b) cases correspond to parameters from Table \ref{table:2}(Test cases 1,4). The (c) correspond  to the real obstacle test case. }
\label{Comp_1}
\end{minipage}
\hfill
\begin{minipage}[b]{0.45\textwidth}
\includegraphics[width=\textwidth]{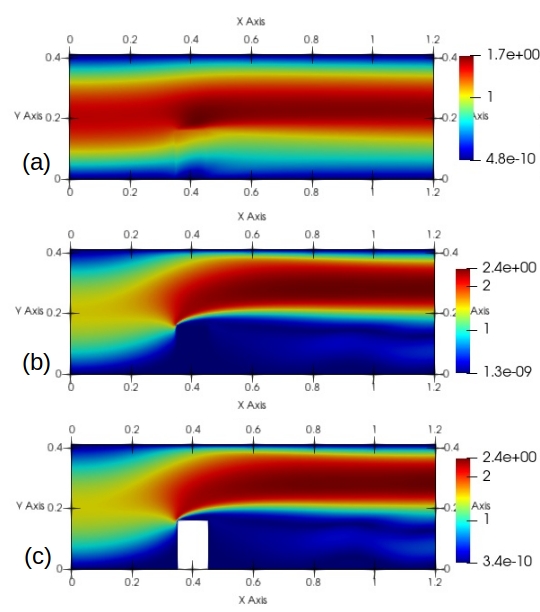}
\caption{The magnitude of the velocity for the wall test problem. The (a), (b) cases correspond to parameters from Table \ref{table:2}(Test case 2, 5). The (c) correspond to the real obstacle test case. }
\label{Comp_2}
\end{minipage}
\end{figure}

\begin{figure}[!tbp]
\centering
\begin{minipage}[b]{0.8\textwidth}
\includegraphics[width=\textwidth]{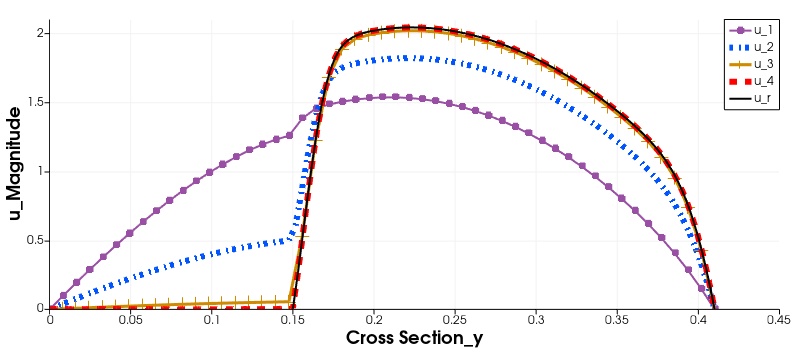}
\caption{Comparison of the velocity magnitude for the half ball case. The velocity $u_k$ correspond to test case number $k$ from Table \ref{table:2}. }
\label{Comp_3}
\end{minipage}
\end{figure}


\begin{figure}
\centering
\begin{minipage}[b]{0.8\textwidth}
\includegraphics[width=\textwidth]{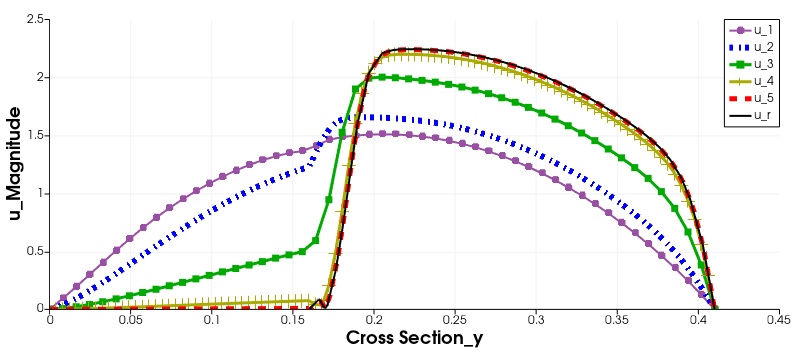}
\caption{Comparison of velocity magnitude for the wall case. The velocity $u_k$ correspond to test case number $k$ from Table \ref{table:2}. }
\label{Comp_5}
\end{minipage}
\end{figure}  

\begin{table}[h!]
\centering
\begin{tabular}{ |p{1.0cm}|p{1.5cm}| p{1.5cm}| }
\hline
Test case & Viscosity $m$ of the half ball domain  &Viscosity $m$ of the wall domain \\
\hline
1 & 10 & 10  \\
2 & 100 & 100  \\
3 & $10^3$ & $10^3$  \\
4 & $10^4 $& $10^4 $ \\
5 & - & $10^5 $ \\
\hline
\end{tabular}
\caption{Table corresponds to Fig \ref{Comp_3}-\ref{Comp_5}.}
\label{table:2}
\end{table}
  
We compare solutions in a cross-section $y$ at $x=0.4$, along the vertical axis of symmetry of the obstacle. The comparison graphs of extracted solution of each test are given in Fig.\ref{Comp_3}-\ref{Comp_5}, where the velocity $u_k$ in graphs correspond to the test case number $k$ from Table \ref{table:2}.
The Fig.\ref{Comp_3}-\ref{Comp_5} show that the solution of the approximate problem approaches the solution of the real obstacle problem relatively fast.
The solution of the approximate problem in the region $S$ is really close to the solution of the rigid obstacle problem as predicted in Theorem \ref{limit}.  

 These results establish that the regularized solutions can
approximate the limiting case reasonably well even for small-scale penalty parameter. 
 Moreover, we show that our approach has practical application in numerics.       

\section{Appendix}
\subsection{Proof of Theorem 2.1}
\begin{proof}
The existence of $u$ is proved by Galerkin method (Temam [\ref{temam}]): we construct the approximate solution of (\ref{weakNS}) and then pass to the limit.

From the definition of $V$ space there exists a sequence $w_1,w_2,...,w_n...$ of linearly independent elements of $\mathcal{V}$ which is total in $V$.
 For each $n$ we define an approximate solution $u_n$ of (\ref{weakNS}) by

\begin{equation}\label{galerkin}
 u_n=\sum_{j=1}^{n}c_{jn} w_{j}\quad 
\end{equation}
with unknown coefficients $c_{jn}\in \mathbb{R} $, satisfying  
\begin{equation}\label{NSgalerkin}
(\nu \mathbf{D}u_n,\mathbf{D}w_j)+ b(u_n,u_n,w_j)=\langle f(t),w_j\rangle
\end{equation}  
 for $j= 1,2...,n$.
The equations (\ref{galerkin}) and (\ref{NSgalerkin}) are a system of nonlinear equations for $c_{1n},...,c_{nn}$, the existence of a solution of this system follows from the Lemma 1.4 ([\ref{temam}] Ch.II, Lemma 1.4.), that is consequence of the Brouwer Fixed Point Theorem: 
\begin{lemma}
Let $X$ be finite dimensional Hilbert space with scalar product $\left[\cdot ,\cdot\right] $ and norm $\left[ \cdot\right] $ and let $P$ be a continuous mapping from $X$ into itself such that for some $k >0$
\begin{equation*}
\left[ P(\xi),\xi \right]  > 0 \,\, for \,\,\left[ \xi \right] = k .  
\end{equation*}
Then there exists $\xi \in X , \, \left[ \xi\right] \leqslant k $, such that
\begin{equation*}
P(\xi )=0.
\end{equation*}
\end{lemma} 
We apply this lemma for proving the existence of $u_{n}$ as follows:

Let $X$ be the space spanned by $w_1,w_2,...,w_n$; the scalar product on $X$ is the scalar product induced by $V $, and $P=P_{m}$ is defined by
\[ \left[P_{m}((u),v \right]=(\nu\mathbf{D}u,\mathbf{D}v)+b(u,u,v)-(f,v),\,\forall\, u,v \in X.  \]
Let us check that scalar product $[\cdot,\cdot]$ is positive
\[\left[P_{n}(u),u\right]=\Vert\nu \mathbf{D}u\Vert^2+b(u,u,u)-(f,u) \\ =\Vert\nu \mathbf{D}u\Vert^2-(f,u)\\
\geqslant c(\nu)\Vert u\Vert^2-\Vert f\Vert \Vert u \Vert.\]
In the last inequality were used the Korn and Cauchy-Schwartz inequalities. Therefore
\begin{equation}
\left[P_{n} (u),u \right]\geqslant \Vert u \Vert(c(\nu)\Vert u \Vert-\Vert f\Vert)
.\end{equation}
It follows that $\left[ P_{n} u,u\right] > 0$ for $\Vert u \Vert=k$, and $k>\frac{1}{c(\nu)}\Vert f \Vert$. It follows that, there exists a solution $\mathbf{u}_{n}$ of (\ref{galerkin})-(\ref{NSgalerkin}).
We multiply (\ref{NSgalerkin}) by $c_{jn}$, this gives
\[\Vert\nu \mathbf{D}u_{n}\Vert_{L^2}^2 +b( u_n ,u_n ,u_n )=(f,u_n).\]
We know that, trilinear form  $b(u_n ,u_n ,u_n )=0$, and get

\begin{equation}\label{exisEst}
 \Vert\nu \mathbf{D}u_{n}\Vert_{L^2}^2 =(f,u_n )\leqslant \Vert f\Vert_{L^2} \Vert u_n \Vert_{L^2} .
\end{equation}

By Korn inequality
we get a priori estimate
\begin{equation}
\Vert u_n  \Vert_{V}\leqslant \Vert f \Vert_{L^2} 
.\end{equation}

Hence the sequence remains bounded in $V$, there exists a subsequence $k\longrightarrow \infty$ such that 
$u_{n_k}\rightharpoonup u$ in $V$.
From compact embedding of $V\hookrightarrow L^2$, so we have also 
$u_{n_k}\longrightarrow u$ in $L^{2}(\Omega)$.\\
If  $u_n$ converges to $u$ in  $W^{1,2}$ weakly and  in $L^2$ strongly, then we need to show that
\[b(u_n,u_n,v)\longrightarrow b(u,u,v),\,\,\forall v\in V .\]

Then we can pass to the limit in (\ref{NSgalerkin}) with the subsequence $k\longrightarrow \infty$, we find that
\begin{equation*}
(\nu \mathbf{D}u,\mathbf{D}v)+ b(u,u,v)=\langle f(t),v\rangle
\end{equation*}
for any $v= w_1,...,\, w_n...$. The above equation is also true for any $v$ which is the linear combination of $w_1,...,\, w_n...$. Since this combination are dense in $V$, a continuity argument shows that the above equation holds for each $v\in V$ and that $u$ is a solution of (\ref{weakNS}).

From the properties of trilinear form we have 

  \[b(u_m ,u_m ,v)=-b(u_m ,v,u_m)=-\sum_{i,j=1}^{n}\int_{\Omega} u_{mi} u_{mj}\partial_{x_i }v_j\,dx.\]
  We know that $u_{mi}\rightarrow u_{i}$ converges strongly in $L^{2}(\Omega)$, since $\partial_{xi}v\in L^{\infty}(\Omega)$, so we have
  \[\int_{\Omega} u_{mi}u_{mj}\partial_{x_i }v_{j}\rightarrow \int_{\Omega} u_{i}u_{j}\partial_{x_i}v_{j}\,dx.\]
  Hence $b(u_m ,v,u_m )$ converges to $b(u,v,u)=-b(u,u,v)$.
  
\end{proof}
\begin{proof}[Proof of Theorem 2.2]
From (\ref{exisEst}) we could deduce that for the domain $S$ we get 
\begin{equation}
m\int_{S}\vert \mathbf{D}(u)\vert^2 \,d x \leqslant \Vert f\Vert_{L^2} \Vert u \Vert_{L^2} 
.\end{equation}
 That is, if $m\longrightarrow \infty$, then $\mathbf{D}u\longrightarrow 0$ in $S$. Therefore, in the limit, as $m\longrightarrow \infty$, we obtain the rigid motion. 

\end{proof}

\subsection{ Bogovskii type estimate}
 
	We need to derive some energy estimates for Stokes system of equations for the (\ref{stokes_eq}) case, to show that the velocity vector field is in Hilbert space.
	 
From Korn inequality (\ref{eq:korn_eq}) for Stokes system (\ref{stokes_eq}) we have:
\begin{equation}
\int_{\Omega} \nu (x)\vert\mathbf{D} u \vert^2 dx \geqslant c\int_{\Omega} \vert\nabla u\vert^2 dx
.\end{equation}

 Using H\"older and Poincar\'e inequalities we get:
\begin{equation}\label{2.2}
\int_{\Omega} \nu (x)\vert\mathbf{D} u \vert^2 dx \leqslant  \Vert f\Vert_{L^2} \Vert u\Vert_{L^2}\leqslant c_p\Vert f\Vert_{L^2}\Vert \nabla u\Vert_{L^2}
.\end{equation} 
We get energy estimate:
\begin{equation}\label{2.3}
\Vert \nabla u\Vert_{L^2} \leqslant  C_1\Vert f\Vert_{L^2}
.\end{equation}

The Definition \ref{def_Stokes} has no information about the pressure field. Since $u$ is a weak solution, we know from [\ref{temam}, Lemma 2.1] that there exists $p\in L^2(\Omega)$ such that 
\begin{equation}\label{eq:2.13_eq}
(\nu\mathbf{D} u,\mathbf{D}\psi)=-(f,\psi)+(p,{\rm div\,}\psi)
\end{equation}
holds for $\psi\in C_0^{\infty}(\Omega)$.
So, to every weak solution we are able to associate a pressure $p$ in such a way that \eqref{eq:2.13_eq} holds.  We formulate the following result for our case. 
\begin{lemma}\label{lem_bogovskii} Let $\Omega \subset \mathbb{R}^2 $ be an open bounded domain of class $C^2$, and let $f\in L^2 (\Omega)$. A vector field $u\in W^{1,2} (\Omega)$ satisfies \eqref{eq:1.6_eq} for all $\phi\in \mathcal{V}(\Omega)$ if and only if there exists a pressure $p\in L^2(\Omega)$ such that \eqref{eq:2.13_eq} holds for every $\psi\in C_0^{\infty} (\Omega)$. Then
\[p\in L^2 (\Omega).\]
Finally, if we normalize $p$ by the condition
\begin{equation}\label{eq:comp_eq}
\int_{\Omega}p=0
.\end{equation}
The following estimate holds  
\begin{equation}\label{eq:bogov1}
\Vert p\Vert_{L^2}\leqslant c_b\left( \Vert\nu \mathbf{D}u\Vert_{L^2}+\Vert f\Vert_{L^2}\right)
.\end{equation} 
\end{lemma}
\begin{proof} The existence of pressure $p$ follows from Temam ([\ref{temam}], Lemma 2.1). 
Let us consider the functional  
\[\mathcal{F}(\psi)=(\nu \mathbf{D}u,\mathbf{D}\psi)+(f,\psi)\]
for $\psi \in H^{1}_0(\Omega).$
 By assumption, $\mathcal{F}$ is bounded in $H^1_0 (\Omega)$ and is identically zero in $\mathcal{V}(\Omega)$. Consider the problem
\[{\rm div\,}\psi= p\]
\begin{equation}\label{eq:sys_eq}
\psi \in H^{1}_0 (\Omega)
\end{equation}
\[\Vert\psi\Vert_{H^{1}}\leqslant c_1 \Vert p\Vert_{L^2}\]
with $\Omega$ bounded and satisfying the cone condition. Since
\[\int_{\Omega}p=0,\]
from Theorem III.3.1 (in [\ref{galdi_book}]) we deduce the existence of $\psi$ solving \eqref{eq:2.13_eq}. If we replace such a $\psi$ into \eqref{eq:2.13_eq} and use \eqref{eq:comp_eq} together with the H\"{o}lder inequality and Poincar\'e inequality we have
\begin{equation}
\begin{split}
\Vert p\Vert_{L^2}^2 \leqslant\Vert\nu\mathbf{D}u\Vert_{L^2}\Vert\mathbf{D} \psi\Vert_{L^2}+\Vert f\Vert_{L^2} \Vert \psi \Vert_{L^2}\\
\leqslant\Vert\nu\mathbf{D}u\Vert_{L^2}\Vert\nabla \psi\Vert_{L^2}+\Vert f\Vert_{L^2} \Vert \psi \Vert_{L^2} \leqslant c_1 \Vert \nu\mathbf{D}u\Vert_{L^2}\Vert p\Vert_{L^2}+c_2 \Vert f\Vert_{L^2} \Vert p \Vert_{L^2}
,\end{split}
\end{equation}

\[\Vert p\Vert_{L^2} \leqslant c\Vert \nu\mathbf{D}u\Vert_{L^2}+c \Vert f\Vert_{L^2} \]
 from \eqref{eq:1.6_eq} and Young's inequality we obtain (\ref{eq:bogov1}). 
  The proof is therefore completed.
\end{proof}

\subsection{Korn Inequality}
\begin{lemma}[Korn Inequality]
Let $\Omega \subset \mathbb{R}^2 $ be an open bounded domain of class $C^2$. Then there exists constant $c>0$ such that 
\begin{equation}\label{eq:korn_eq}
\int_{\Omega} \nu (x) \vert\mathbf{D} u \vert^2 dx \geqslant c_1(\nu_{0}) \Vert u\Vert^2_{H^1 (\Omega)}  
\end{equation} 
for all $u\in V$.
\end{lemma}
\begin{proof} The proof is based on the results from [\ref{mucha}, Lemma 2.1]. We will prove it in general case where viscosity is bounded $\nu(x)\geq \nu_{0}$, where $\nu_{0}$ is a fixed constant. 
Without loss of generality we could assume that $u\in C^{2}_{c} (\Omega) $. Consider,   
\begin{equation}\label{eq:korn}
\int_{\Omega} \left( \mathbf{D}u\right) ^2 dx=\frac{1}{2}\int_{\Omega}\sum_{i,j=1}^{2}(u^i_{,j}+u^j_{,i})^2 dx=\Vert\nabla u\Vert_{L^2}^2+\int_{\Omega}\sum_{i,j=1}^{2}u^i_{,j}u^j_{,i} dx
.\end{equation}

Integration by parts of the last term on the RHS of (\ref{eq:korn_eq}) gives us

\begin{equation}\label{eq:2.2_eq}
\int_{\Omega}\sum_{i,j=1}^{2}u^i_{,j}u^j_{,i}=\int_{\Omega}\sum_{i,j=1}^{2}u^i_{,i}u^j_{,j}+\int_{\partial \Omega}\sum_{i,j=1}^{2}\left(u^i_{,j}u^j n_i -u^i_{,i}u^j n_j \right) 
.\end{equation}
As $u$ is compactly supported in $\Omega$, it follows that the second term of the RHS of (\ref{eq:2.2_eq}) is zero, so we get (\ref{eq:korn}). 
 
\end{proof}

\textbf{Acknowledgements.} The author would like to thank prof. Piotr Mucha, dr. Piotr Krzyzanowski and dr. Tomasz Piasecki for their invaluable help and remarks through the process of creating the paper.

\end{document}